\documentclass[11pt, reqno]{amsart}%
\usepackage{color,cancel,pgfplots,mathabx,mathtools}

\usepackage[normalem]{ulem}
\usepackage{hyperref}

\usepackage{bbm}
\usepackage{enumerate}
\usepackage{amstext}
\usepackage{mathrsfs}
\usepackage{xspace}
\usepackage{amsfonts}
\usepackage{amsmath}
\usepackage{amssymb}
\usepackage{amstext}
\usepackage{amsthm}       
\usepackage{xspace}
\usepackage[active]{srcltx}
\usepackage{cite}

\usepackage{mathrsfs} 

\DeclareMathOperator*{\esssup}{ess\,sup}
\DeclareMathOperator*{\essinf}{ess\,inf}

\newcommand{\HV}{{\mathcal H}_{\beta }}
\newcommand{\HVtilde}{{\mathcal H}_{\tilde \beta}}
\newcommand{\HVminus}{\mathcal{H}_{\beta_-}}
\newcommand{\HVzero}{{\mathcal H}_{0}}
\newcommand{\HVone}{{\mathcal H}_{\beta_1 }}
\newcommand{\HVtwo}{{\mathcal H}_{\beta_2 }}

\newcommand{\CC}{\D{C}}

\newcommand{\D}[1]{\mbox{\rm #1}}
\newcommand{\cal}{\mathcal}

\newcommand{\F}{\mathcal F}
\newcommand{\G}{\mathcal G}
\newcommand{\R}{\mathbb R}

\newcommand{\Lip}{\D{Lip}}

\newcommand{\N}{\mathbb N}

\newcommand{\PP}{\mathbb{P}}

\newcommand{\Z}{\mathbb Z}

\newcommand{\sol}{\mathfrak{S}}

\newcommand{\comp}{\mbox{\scriptsize  $\circ$}}
\newcommand{\eps}{\varepsilon}
\renewcommand{\epsilon}{\varepsilon}

\newcommand{\conv}{\mathrm{conv}}

\newcommand{\cyl}{(0,+\infty)\times\R}
\newcommand{\ccyl}{[0,+\infty)\times\R}

\newcommand{\cTcyl}{[0,T]\times\R}

\newcommand{\e}{\textrm{\rm e}}

\newcommand{\tagliato}{$\kern-5 mm -$}
\newcommand{\tagliat}{$\kern-4 mm -$}

\newcommand{\cchi}{\mbox{\large $\chi$}}

\newcommand{\B}[1]{\mbox{\boldmath $#1$}}

\newcommand{\Ham}{\mathscr{H}(\gamma,\alpha_0,\alpha_1)}

\newtheorem{teorema}{Theorem}[section]
\newtheorem{prop}[teorema]{Proposition}
\newtheorem{lemma}[teorema]{Lemma}
\newtheorem{definition}[teorema]{Definition}
\newtheorem{cor}[teorema]{Corollary}

\newtheorem{guess}[teorema]{Remark}
\newtheorem{example}[teorema]{Example}

\newenvironment{oss}{\begin{guess} \begin{rm}}{\end{rm} \end{guess}}
\newenvironment{examp}{\begin{example} \begin{rm}}{\end{rm} \end{example}}

\numberwithin{equation}{section}
\begin{document}

\title{
  Stochastic homogenization of a class of nonconvex viscous HJ equations in one space dimension
  }

\author{Andrea Davini \and Elena
  Kosygina} \address{Dip. di Matematica, {Sapienza} Universit\`a di
  Roma, P.le Aldo Moro 2, 00185 Roma, Italy}

\email{davini@mat.uniroma1.it}
\address{Department of Mathematics, Baruch College, One Bernard Baruch Way, Box B-630, New York, NY 10010, USA}

\email{elena.kosygina@baruch.cuny.edu}

\begin{abstract}
  We prove homogenization for a class of nonconvex (possibly
    degenerate) viscous Hamilton-Jacobi equations in stationary
  ergodic random environments in one space dimension. The results
  concern Hamiltonians of the form $G(p)+V(x,\omega)$, where the
  nonlinearity $G$ is a minimum of two or more convex functions with
  the same absolute minimum, and the potential $V$ is a bounded
  stationary process satisfying an additional {scaled hill and
    valley condition.  This condition is trivially satisfied in the
    inviscid case, while it is equivalent to the original hill and
    valley condition of A.\,Yilmaz and O.\,Zeitouni \cite{YZ19} in the
    uniformly elliptic case.} Our approach is based on PDE methods and
  does not rely on representation formulas for solutions.  Using only
  comparison with suitably constructed super- and sub- solutions, we
  obtain tight upper and lower bounds for solutions with linear
  initial data $x\mapsto \theta x$.  Another important ingredient is a
  general result of P.\,Cardaliaguet and P.\,E.\,Souganidis
  \cite{CaSo17} which guarantees the existence of sublinear correctors
  for all $\theta$ outside ``flat parts'' of effective Hamiltonians
  associated with the convex functions from which $G$ is built. We
  derive crucial derivative estimates for these correctors which allow
  us to use them as correctors for
  $G$. 
\end{abstract}


\date{\today}
\maketitle
\section{Introduction}
%
%
We are interested in proving a homogenization result as
$\epsilon\to 0^+$ for a viscous Hamilton-Jacobi (HJ) equation of the form
\begin{equation}\label{intro viscous hj}
  \partial_{t }u^\epsilon=\epsilon a\left(\frac{x}{\epsilon},\omega\right) \partial^2_{xx} u ^\epsilon+G(\partial_x u^\epsilon)+  \beta V\left(\frac{x}{\epsilon},\omega\right),  \qquad (t,x)\in (0,+\infty)\times\R,
\end{equation}
where $G:\R\to\R$ belongs to a certain class of continuous, nonconvex
and coercive functions. Dependence on a realization of random
environment $\omega$ enters through the diffusion coefficient
$a(x,\omega)$ and potential $V(x,\omega)$ which are assumed to be
stationary with respect to shifts in $x$ and Lipschitz continuous with
a constant independent of $\omega$. Moreover, we suppose that $a$ and
$V$ take values in $[0,1]$ and with probability 1
\begin{equation}\label{intro 01}
 \essinf_{x\in\R} {V}(x,\omega)=0\qquad\hbox{and}\qquad  \esssup_{x\in\R} {V}(x,\omega)=1.
\end{equation}
Thus, the parameter $\beta\ge 0$ represents the ``magnitude'' of the
potential $V$. For a complete set of conditions on the coefficients
and precise statements of our results, we refer to Section~\ref{sez
  intro main}.
%

We shall say that {\em the equation \eqref{intro viscous hj}
  homogenizes} if there exists a continuous function $\HV(G):\R\to\R$
called {\em effective Hamiltonian} and a set $\Omega_0$ of
probability 1 such that for every $\omega\in\Omega_0$ and every
uniformly continuous function $g$ on $\R$, the solution $u^\eps$ of
\eqref{intro viscous hj} satisfying $u^\epsilon(0,\cdot,\omega)=g$
converges locally uniformly on $\ccyl$ as $\eps\to 0^+$ to the unique
solution $\overline u$ of the (deterministic) {\em effective equation}
\begin{equation}\label{intro eq effective}
\partial_t\overline u+\HV(G)(\partial_x\overline u)=0\qquad\hbox{in $\cyl$}
\end{equation}
satisfying $\overline u(0,\cdot)=g$.
Solutions to all PDEs considered in this paper are understood in the
viscosity sense. We refer the reader to \cite{barles_book, bardi,
  users} for details on viscosity solution theory.

To put our results in a broader context, we shall first briefly
review the existing literature on non-convex homogenization of
viscous HJ equations.

\subsection{Literature review.}\label{backgr}

Equation \eqref{intro viscous hj} belongs to a general class of viscous HJ
equations of the form
\begin{equation}
  \label{intro HJeps}
  \partial_{t}u^\epsilon=\epsilon\,\mathrm{tr}\left(A\left(\frac{x}{\epsilon},\omega\right)D^2_{xx}u\right)+H\left(D_{x}u,\frac{x}{\epsilon},\omega\right),\quad(t,x)\in(0,+\infty)\times\R^d,
\end{equation}
where the non-negative definite diffusion matrix $A(x,\omega)$ and the
Hamiltonian $H(p,x,\omega)$ are stationary under the shifts by
$x\in\R^d$ and satisfy some regularity and growth assumptions.  

For homogenization results concerning viscous HJ equation \eqref{intro HJeps}
with {\em convex} (with respect to $p$) Hamiltonians in
the stationary ergodic setting under various sets of assumptions we
refer the reader to \cite{LS_viscous, KRV, KV, LS_revisited, AS12,
  AT15, AC15, JST} and references therein.

Recently it was shown by counterexamples for
$H(p,x,\omega)=G(p)+V(x,\omega)$, first for inviscid (i.e. with $A\equiv 0$) HJ equations,
\cite{Zil, FS}, and then also for viscous HJ equations with
$A\equiv \text{const}$, \cite{FFZ}, that in two or more space
dimensions a {\em strict local saddle point} of 
$G$ and a specially ``tuned'' potential in a
{\em very slowly mixing\footnote{polynomially mixing of order 1,
    \cite[Section 3.1]{Zil}}} random environment can prevent
homogenization.  It is not known whether the absence of saddle points
and/or fast mixing (or even finite range dependence) conditions on the
environment would allow to get a general homogenization result. To
date, there exist several classes of examples of homogenization for HJ
equations with non-convex Hamiltonians in the stationary ergodic
setting for all $d\ge 1$, \cite{ATY_nonconvex, AS, AC18, CaSo17, FS,
  QTY, Gao19}, but an overall picture is far from being
complete. Among these examples the viscous case is considered only in
\cite{AC18} and \cite[Corollary 3.9]{CaSo17}. Key assumptions in
the last two references which facilitate homogenization are:
\begin{itemize} 
\item [\cite{AC18}:] homogeneity of degree
  $\alpha>1$ of the Hamiltonian with respect to $p$;
\item [\cite{CaSo17}:] homogeneity of degrees 0 and 1 in $p$ of the
  diffusion matrix $A(p,x,\omega)$ and Hamiltonian $H(p,x,\omega)$
  respectively and radial symmetry of the joint law of $(A,H)$.
\end{itemize}
We refer to the original papers for precise statements. 

However, for $d=1$, equations of the form \eqref{intro HJeps} with
$A\equiv 0$ in stationary ergodic environments are known to homogenize
without any additional mixing conditions, \cite{ATY_1d,Gao16}. A
cornerstone tool used in these papers is the homogenization result for
level-set convex Hamiltonians,\cite{AS}. The last result covers all
$d\ge 1$. Its proof crucially uses the assumption that the original
equation is of the first order and does not extend to the viscous
case.

Nevertheless it is hard to imagine that addition of a viscous term
(especially a uniformly elliptic $A$) can turn a homogenizable HJ
equation into a non-homogenizable one (under a standard set of
assumptions). Thus, further attempts are necessary to resolve the
issue even in the one-dimensional case.

For $d=1$, apart from already mentioned works \cite{AC18,CaSo17},
there are other classes of examples of homogenization for viscous HJ,
\cite{DK17, YZ19, KYZ20}. In \cite[Section 4]{DK17} the authors have
shown homogenization of \eqref{intro HJeps} with $H(x,p,\omega)$ which
are ``pinned'' at one or several points on the $p$-axis and convex in
each interval in between. For example, for every $\alpha>1$ the
Hamiltonian $H(p,x,\omega)=|p|^\alpha-c(x,\omega)|p|$ 
is pinned at $p=0$ (i.e.\
$H(0,x,\omega)\equiv
\mathrm{const}$) 
and convex in $p$ on each of the two intervals $(-\infty, 0)$ and
$(0,+\infty)$.
%

Clearly, adding a non-constant potential breaks the pinning
property. In particular, homogenization of equation \eqref{intro
  HJeps}, where $d=1$, $A\equiv \mathrm{const}>0$,
\begin{equation}
  \label{open}
  H(p,x,\omega):=\frac12\,|p|^2-c(x,\omega)|p|+\beta V(x,\omega),\quad
  0<c(x,\omega)\le C,\quad \text{and }\ \beta>0
\end{equation}
remained an open problem even when $c(x,\omega)\equiv c>0$.  The
authors of \cite{YZ19} introduced a novel hill and valley condition on
$V$ (see ($\Lambda$V) in Subsection~\ref{oss weak valley-and-hill})
and were able to handle the case $c(x,\omega)\equiv \mathrm{const}>0$
in the discrete setting of controlled random walks in a random
potential on $\Z$. This work paved out the way for \cite{KYZ20} which
gave a proof of homogenization for \eqref{intro HJeps} with
$A\equiv 1/2$ and $H$ as in \eqref{open} with $c(x,\omega)\equiv c>0$,
retaining the hill and valley condition.
The case when both $c(x,\omega)$ and $V(x,\omega)$
in \eqref{open} are non-constant is still open.

 {While the hill and valley condition clearly excludes the classical
periodic case, it holds true for a large and representative class of
stationary ergodic environments ranging from those with finite range
of dependence or exponentially mixing to very slowly mixing or even
non-mixing. In the realm of stationary ergodic media, periodic as well
as almost-periodic environments constitute a very important but
  also a very special sub-class treatable by methods based on
compactness. Loss of compactness is considered to be one of the main
challenges in dealing with general stationary ergodic media. From this
point of view, stationary ergodic potentials which satisfy the hill
and valley condition can be considered typical, as we argue in
Appendix~\ref{hv} and to which we refer for further discussion
and examples. It would certainly be desirable to drop this condition
altogether but, given a relatively slow progress in the viscous case
in comparison to the inviscid one, the hill and valley condition,
  a relaxed version of which we retain in this paper, allows us
  to move forward without imposing any mixing conditions on the
  environment (in contrast with the widely accepted in the literature
  finite range dependence case).}

\subsection{Discussion of the main results.} The current paper
presents new results on homogenization of \eqref{intro viscous hj}
with non-convex $G$ which considerably extend those in \cite{
  KYZ20}. Moreover, it gives a much simpler proof which does not rely
on Hopf-Cole transformation or stochastic control representation of
solutions and is based solely on PDE techniques.   {We also replaced the
{\em hill and valley condition} ($\Lambda$V) on the
    potential $V$ (see Appendix~\ref{hv}) imposed in \cite{KYZ20}
  with a weaker {\em scaled hill and valley condition} (V2) (see
  Section \ref{sez intro main}). The two conditions are equivalent in
the uniformly elliptic case, i.e.\ when the diffusion coefficient
  $a(x,\omega)$ is bounded away from $0$, while in the inviscid case
(V2) reduces to \eqref{intro 01} and, thus, does not add any
  additional restrictions. We refer to Subsection~\ref{oss weak
    valley-and-hill} for further details.}


Let us recall that \cite{KYZ20} considered the equation \eqref{intro
  viscous hj} where $a\equiv 1/2$ and
\begin{equation}
  \label{quad}
  G(p)=(G^+\wedge G^-)(p)=\frac12\,|p|^{2}-c|p|=\min\left\{\frac12\,|p|^{2}-cp,
  \frac12\,|p|^{2}+cp\right\}
\end{equation}
assuming that the potential $V$ is sufficiently regular, satisfies
\eqref{intro 01} and the already mentioned hill and valley
condition. Theorem \ref{teo main 1} of our paper (see Section~\ref{sez
  intro main}) establishes homogenization for \eqref{intro viscous hj}
with a (possibly degenerate) Lipschitz continuous diffusion
coefficient $a:\R\times\Omega\to [0,1]$ and $G=G^+\wedge G^-$, where
$G^\pm$ are convex and coercive functions with $\min G^+=\min G^-$
satisfying fairly general assumptions. Theorem \ref{teo main 2}
extends this result to $G$ which is the minimum of any finite number
of such functions as long as all of them have the same absolute
minimum.  {The assumptions on $V$ are essentially the same as in
  \cite{KYZ20} except that the hill and valley condition ($\Lambda$V) is
  replaced with (V2).}


Even though our general strategy is analogous to that of \cite{KYZ20},
the technical realization is different and includes significant
shortcuts. Just as in \cite{DK17, KYZ20}, an application of \cite[Lemma 4.1]{DK17} reduces the proof of homogenization to showing that 
for every $\theta\in\R$
\begin{eqnarray}\label{intro hom}
\HV^L(G) (\theta):=\liminf_{\eps\to 0^+}\ {u^\eps_\theta(1,0,\omega)}
= 
\limsup_{\eps\to 0^+}\ {u^\eps_\theta(1,0,\omega)}
=:\HV^U({G})(\theta) \quad \PP\text{-a.s.,}
\end{eqnarray}
where $u^\epsilon_\theta$ is the solution of \eqref{intro viscous hj}
with initial condition $u^\epsilon_\theta(0,x,\omega)=\theta x$.  As
in \cite{KYZ20}, we first establish tight upper and lower bounds for
the deterministic functions $\HV^L(G), \HV^U(G)$ defined above.
This is obtained by constructing suitable sub- and super- solutions
for equation \eqref{intro viscous hj} and by comparing them with the
solutions $u^\eps_\theta$, where we only exploit well known comparison
principles and Lipschitz estimates for solutions of \eqref{intro
  viscous hj}.  The proof does not depend on explicit formulas and
does not involve stochastic analysis. It is technically much
simpler than that in \cite{KYZ20}.

The proof of \eqref{intro hom} for $\theta$ outside the intervals
where the effective Hamiltonian is constant depends on construction of
sublinear correctors associated with $G^\pm$ and on establishing
suitable gradient bounds for these correctors, which allow us to use
them as correctors associated with $G$.  In \cite{KYZ20}, such
properties were established by direct computation since, due to the
special form of the nonlinearity in \eqref{quad}, the authors were
able to represent the correctors via the Feynman-Kac formula. In our
more general setting, the existence of sublinear correctors for
$G^\pm$ follows from a recent result of P.\,Cardaliaguet and
P.\,E.\,Souganidis \cite{CaSo17}, while the bounds on their
derivatives are consequence of suitable comparison principles for the
associated viscous HJ equation that we prove in the Appendix~\ref{pderes}.  The
construction in \cite{CaSo17} provides sublinear correctors 
  which, in general, are not expected to have stationary gradient.
Nevertheless, this is true here and it is due to the fact that
sublinear solutions of the corresponding viscous HJ equation are
unique up to additive constants, as we show in the Appendix~\ref{pderes}. This
remark is included in the statement of Proposition \ref{prop CS}, even
though this stationarity property is not used in our proof of 
the homogenization result. 

Our second result, Theorem~\ref{teo main 2}, extends this
homogenization result to $G$ which is the minimum of three or more
convex functions with same absolute minimum.  The argument is new. It
is based on the crucial remark that if $G$ is the minimum of two convex functions with same absolute minimum, then homogenization commutes with convexification, see Section \ref{sez teo main 2}.

\subsection{Outline of the paper.} Precise conditions and statements
of the main results, Theorem~\ref{teo main 1} and Theorem~\ref{teo
  main 2}, are given in Section~\ref{sez intro
  main}. Section~\ref{prelim} presents several basic facts which are
used throughout the paper. Upper and lower bounds on the effective
Hamiltonian are derived in Section~\ref{bounds}. Section~\ref{sez
  existence of correctors} is devoted to construction of sublinear
correctors and derivative estimates. The proofs of the two main
theorems are given in Sections~\ref{sez teo main 1} and \ref{sez teo
  main 2}. The necessary PDE results are collected in the
Appendix~\ref{pderes}. Appendix B discusses the original and
  scaled hill and valley conditions in more detail and shows that they
  are satisfied for a wide range of typical stationary ergodic
  environments.

\begin{oss}
  Below we sometimes refer to ``known results in stationary ergodic
  homogenization''. The results we have in mind are for convex
  Hamiltonians. They are contained in many papers cited at the
  beginning of Section~\ref{backgr}. However, it is probably most
  convenient to refer to \cite{AT15} {if
    necessary}, as all our assumptions are satisfied in the setting of
  \cite{AT15}.
\end{oss}

\noindent{\bf Acknowledgements.} 
The definition of the scaled hill and valley condition, as it appears
in this paper, originates from a similar one proposed by Atilla Yilmaz
in discussions with the authors in summer of 2020.

  The authors were partially supported by a Collaboration Grant for
  Mathematicians \#523625 from the Simons Foundation. The second
  author thanks ``Visiting Professor Programme'' of Sapienza
  University of Rome for funding and, in particular, the Department of
  Mathematics for inclusiveness and wonderful hospitality. She
  gratefully acknowledges the support of the Fields Institute through
  Fields Research Fellowship and thanks its staff for stimulating
  research environment.

\section{Main results}\label{sez intro main}
Let $\Omega$ be a Polish space, ${\cal F}$ be the $\sigma$-algebra of
Borel subsets of $\Omega$, and $\PP$ be a complete probability measure
on $(\Omega,\F)$. We shall denote by ${\cal B}$ the Borel
$\sigma$-algebra on $\R$ and equip the product space $\R\times \Omega$
with the product $\sigma$-algebra ${\cal B}\otimes {\cal
  F}$. 

We assume that $\PP$ is invariant under the action of a one-parameter
group $(\tau_x)_{x\in\R}$ of transformations
$\tau_x:\Omega\to\Omega$. More precisely, we suppose that the mapping
$(x,\omega)\mapsto \tau_x\omega$ from $\R\times \Omega$ to $\Omega$ is
measurable, $\tau_0=id$, $\tau_{x+y}=\tau_x\comp\tau_y$ for all
$x,y\in\R$, and $\PP\big(\tau_x (E)\big)=\PP(E)$ for every $E\in\F$ and $x\in\R$.
We also require that the action by $(\tau_x)_{x\in\R}$ is {\em
  ergodic,} i.e.\ that any measurable function $f:\Omega\to\R$ such that
$f(\tau_x\omega)=f(\omega)$ a.s.\ in $\Omega$ for every fixed
$x\in{\R}$ is a.s.\ constant.
%

A random process $f:\R\times \Omega\to \R$ is said to be {\em
  stationary with respect to the shifts $(\tau_x)_{x\in\R}$} if  $f(x+y,\omega)=f(x,\tau_y\omega)$ for all $x,y\in\R$ and $\omega\in\Omega$.

Let us consider the unscaled version of \eqref{intro viscous hj} (i.e.\ with $\epsilon=1$) 
\begin{equation}\label{eq general HJ}
\partial_{t }u=a(x,\omega) \partial^2_{xx} u +G(\partial_x u)+  \beta V(x,\omega)  \qquad \hbox{in $(0,+\infty)\times\R$},
\end{equation}
where $a, V:\R\times\Omega\to [0,1]$ are continuous stationary random processes and $V$ satisfies \eqref{intro 01}. 
%
%
We shall also assume that for some $\kappa\in(0,+\infty)$,
\begin{itemize}
\item[(A)]  $\sqrt{a(\cdot,\omega)}:\R\to [0,1]$\quad is $\kappa$--Lipschitz continuous for all $\omega\in\Omega$;\smallskip
\item[(V1)] $V(\cdot,\omega):\R\to [0,1]$\quad is $\kappa$--Lipschitz
  continuous for all $\omega\in\Omega$.
\end{itemize}
In addition, we shall suppose that $V$ under $\PP$ satisfies the {\em scaled hill} (respectively, {\em scaled valley}) {\em condition}:
\begin{itemize}
\item[(V2)]  for every $h\in (0,1)$ and $y>0$ there exists a set $\Omega(h,y)$ of probability 1 such that, for every $\omega\in \Omega(h,y)$, 
there exists $\ell_1<\ell_2$ in $\R$ and $\delta\in (0,1)$ such that 
\begin{itemize}
\item{(a)}\quad$\displaystyle \int_{\ell_1}^{\ell_2} \frac{1}{a(x,\omega)\vee \delta}\,dx = 2y$;
\end{itemize}
and 
\begin{itemize}
\item{(h)}\ {\bf (hill)} \quad $V(\cdot,\omega)\geqslant h\quad\hbox{on $[\ell_1,\ell_2]$}$;\medskip
\end{itemize}
(respectively,\smallskip
\begin{itemize}
\item{(v)}\ {\bf (valley)} \quad $V(\cdot,\omega)\leqslant h\quad\hbox{on $[\ell_1,\ell_2]$}$.)
\end{itemize}
\end{itemize}
Following \cite{KYZ20}, an interval $I$ will be called $h$-hill (resp. $h$-valley) if $V(x,\omega)\geqslant h$ (resp. $V(x,\omega)\leqslant h$) for every $x\in I$.

Next, we introduce the family $\Ham$ of continuous functions
$G:\R\to\R$ satisfying the following conditions, for fixed constants
$\alpha_0,\alpha_1>0$ and $\gamma>1$:
\begin{itemize}
\item[(G1)] \quad $\alpha_0|p|^\gamma-1/\alpha_0\leqslant G(p)\leqslant\alpha_1(|p|^\gamma+1)$\qquad for all $x,p\in\R$;\medskip
\item[(G2)] \quad
  $|G(p)-G(q)|\leqslant\alpha_1\left(|p|+|q|+1\right)^{\gamma-1}|p-q|$  \qquad for all $p,q\in\R$.\medskip
\end{itemize}
The above assumptions guarantee
well posedness in $\D{UC}(\ccyl)$ of the Cauchy problem for
  parabolic equation \eqref{eq general HJ} as well
  as suitable Lipschitz estimates for solutions of  \eqref{eq general HJ} with linear initial data, see
Theorem \ref{teo parabolic eq} and Proposition \ref{prop Lip
  estimates} in Section \ref{sez appendix stationary}. They will be
also used to show that condition (H) in \cite{CaSo17} is fulfilled,
see the proof of Proposition \ref{prop CS}.  We stress
that 
our results hold (with the same proofs) under any other set of
assumptions apt to  ensure the same kind of PDE results. 

Since functions from $\Ham$ are bounded from below in view of (G1), in the sequel without loss of generality we shall always assume 
that $G$ is non-negative. 

As stated  in the introduction, we shall prove homogenization for
the rescaled version \eqref{intro viscous hj} of equation \eqref{eq
  general HJ} for a class of nonconvex functions $G$ in $\Ham$.  With
a slight abuse of terminology, in the sequel we shall say that equation
\eqref{eq general HJ} homogenizes if the rescaled equation
\eqref{intro viscous hj} homogenizes.


For given $c_+\geqslant c_-$ in $\R$, let $G^+,G^-:\R\to [0,+\infty)$ be {\em convex} functions from $\Ham$  with $G^+(c_+)=G^-(c_-)=0$. 
Let us furthermore assume that there exists 
$\widehat{p}\in [c_-,c_+]$ such that 
\begin{equation*}\label{eq F_c}
(G^-\wedge G^+)(p)=G^-(p)\quad \hbox{if $p<\widehat{p}$},
\qquad
(G^-\wedge G^+)(p)=
G^+(p)\quad \hbox{if $p\geqslant \widehat{p}$.}
\end{equation*}
%
%
By well-known results in stationary ergodic
homogenization, 
the equation \eqref{eq general HJ} with $G:=G^{\pm}$ homogenizes and
the effective Hamiltonian $\HV(G^\pm) $ is convex.
We shall prove that equation \eqref{eq general HJ} homogenizes for $G:=G^-\wedge G^+$ as well. The precise statement is given in the next theorem.

\begin{teorema}\label{teo main 1}
  Let $a,V:\R\times\Omega\to [0,1]$ be continuous
    stationary processes satisfying (A), \eqref{intro 01}, (V1)--(V2) and
  $G^+,G^-:\R\to [0,+\infty)$ be convex functions as above. Then the
  viscous HJ equation \eqref{eq general HJ} with $G:=G^-\wedge G^+$
  homogenizes and the effective Hamiltonian $\HV(G^-\wedge G^+)$ can
  be characterized as follows:
\begin{enumerate}[(a)]
\item (Strong potential) if $\beta\geqslant (G^-\wedge G^+)(\widehat{p})$, then 
\begin{equation*}  
\HV (G^-\wedge G^+)(\theta)=
  \begin{cases}
    \HV (G^+)(\theta)&\text{if}\quad \theta> c_+\\
    \beta &\text{if }\quad c_-\leqslant \theta\leqslant c_+\\ 
    \HV (G^-)(\theta)&\text{if}\quad \theta< c_-;
  \end{cases}
\end{equation*}
\smallskip
\item (Weak potential) if $\beta<(G^-\wedge G^+)(\widehat{p})$, then  
\begin{equation*}  
{\HV (G^-\wedge G^+)}(\theta)= 
  \begin{cases}
    \HV (G^+)(\theta)&\text{if}\quad \theta>{\theta} _+\\
    (G^-\wedge G^+)(\widehat{p})&\text{if }\quad{\theta} _{-}\leqslant \theta\leqslant{\theta} _+\\ 
    \HV (G^-)(\theta)&\text{if}\quad \theta< {{\theta} _{-}},
  \end{cases}
\end{equation*}
where ${\theta}_+$ (resp.\ ${\theta}_-$) is the unique
solution in $[\widehat{p}, c_+]$ (resp.\ $[c_-,\widehat{p}]$) of the
equation \[
\HV (G^+)(\theta) =(G^-\wedge G^+)(\widehat{p})\qquad (\hbox{resp. \ \ $\HV (G^-)(\theta) =(G^-\wedge G^+)(\widehat{p})$}\ ).
\]
\end{enumerate}
\end{teorema}

\begin{oss}
  As we shall see, $\HV(G^{\pm})\geqslant \beta$ on $\R$ and
  $\HV(G^{-})(c_-)=\HV(G^{+})(c_+)=\beta$, see Proposition \ref{prop
    Lambda}. Hence, item (a) above amounts to saying that
  $\HV (G^-\wedge G^+)$ is the lower convex envelope of the functions
  $\HV (G^+)$ and $\HV (G^-)$.  ``Convexification'' of the effective
  Hamiltonian in the strong potential case has been already
  observed in the non-viscous case, see \cite{ATY_nonconvex, ATY_1d,
    QTY}.
\end{oss}

Our second  result generalizes Theorem~\ref{teo main 1} to Hamiltonians
which can be represented as a minimum of more than two convex
Hamiltonians. More precisely, let $n\in\N$ with $n\geq 2$ and
$G_0,G_1,\dots,G_n\in{\cal H}(\gamma,\alpha_0,\alpha_1)$ be convex
non-negative functions such that
$G_0(c_0)=G_1(c_1)=\dots=G_n(c_n)=0$
for some $c_0<c_1<\dots<c_n$ and, for each $i\in\{0,1,\dots,n-1\}$,
\[
 (G_i\wedge G_{i+1})(p)=G_i(p)\quad\hbox{if $p< \widehat{p}_{i,i+1}$,}
 \qquad
 (G_i\wedge G_{i+1})(p)=G_{i+1}(p)\quad\hbox{if $p\geq \widehat{p}_{i,i+1}$}
 \]
for some $\widehat{p}_{i,i+1}\in (c_i,c_{i+1})$.

\begin{teorema}\label{teo main 2}
  Let $a,V:\R\times\Omega\to [0,1]$ be continuous stationary
    processes satisfying (A), \eqref{intro 01}, (V1)--(V2), $n\ge 2$, and
  $G_0,G_1,\dots,G_n:\R\to [0,+\infty)$ be convex functions as above. Then the viscous HJ equation \eqref{eq general HJ}
  with $G:=G_0\wedge G_1\wedge\dots\wedge G_n$ homogenizes and the
  effective Hamiltonian $\HV(G_0\wedge G_1\wedge\dots\wedge G_n)$ is
  given by the following formula:
\begin{align}\label{eh0n}
  &\HV(G_0\wedge G_1\wedge\dots\wedge G_n)(\theta)=\min_{i\in\{1,2,\dots,n\}}\HV(G_{i-1}\wedge G_i)(\theta)\\ 
  &\qquad =
  \begin{cases}
    \HV(G_0\wedge G_{1})(\theta),&\text{if }\theta\le c_1;\smallskip\\
    \HV(G_{i-1}\wedge G_i)(\theta),&\text{if }c_{i-1}< \theta\le c_i,\quad i\in\{2,3,\dots,n-1\};\smallskip\\
    \HV(G_{n-1}\wedge G_n)(\theta),&\text{if }\theta> c_{n-1}.
  \end{cases} \label{eh0nl}
\end{align}
\end{teorema}

\begin{oss}
  To avoid repetition, we assume throughout the paper without
    further mention that $a,V:\R\times\Omega\to [0,1]$ are continuous
    stationary processes satisfying (A), \eqref{intro 01}, and (V1).
    Condition (V2) will be imposed only as needed.
\end{oss}

\section{Preliminaries}\label{prelim}
For a given $G\in\Ham$, let us denote by $u_\theta$ the unique Lipschitz solution to \eqref{eq general HJ} with initial condition $u_\theta(0,x)=\theta x$ on $\R$,  
and define the following deterministic quantities, defined almost surely in $\Omega$, see Proposition \ref{prop Lambda} below for the details:
\begin{eqnarray}\label{eq referee}
\HV^U(G) (\theta):=\limsup_{t\to +\infty}\ \frac{u_\theta(t,0,\omega)}{t},\qquad 
\HV^L({G}) (\theta):=\liminf_{t\to +\infty}\ \frac{u_\theta(t,0,\omega)}{t}.
\end{eqnarray}
Observe that, if we denote by $u^\eps_\theta$ the solution of
\eqref{intro viscous hj} with initial condition
$u^\eps_\theta(0,x,\omega)=\theta x$ then we have 
$u^\eps_\theta(t,x,\omega)=\epsilon
u_\theta(t/\epsilon,x/\epsilon,\omega)$. Thus, the above definition of
$\HV^L(G)(\theta)$ and $\HV^U(G)(\theta)$ is consistent with the one
given in \eqref{intro hom}.

In view of \cite[Lemma 4.1]{DK17} and Proposition \ref{prop Lip
  estimates}, in order to prove homogenization it is enough to show
that $\HV^U( G) (\theta)=\HV^L( G) (\theta)$ for every $\theta\in\R$.
In this instance, their common value will be denoted by
$\HV(G)(\theta)$. The function $\HV(G):\R\to\R$ is the effective
Hamiltonian associated to $G$.


The following holds:

\begin{prop}\label{prop Lambda}
  Let $G\in\Ham$. Then the limits
    in \eqref{eq referee} above are almost surely
    constant and, moreover,
  \begin{enumerate}[ (i)]
 \item $\HV^U(G)  (\theta)\geqslant \HV^L(G)(\theta)\geqslant \alpha_0|\theta|^\gamma-1/\alpha_0$\quad for all $\theta\in\R$;\smallskip 
 \item if $V$ satisfies the scaled hill condition (V2)(a),(h) then
   $ \HV^L(G) (\theta)\geqslant \beta$ for every
   $\theta\in\R$;\smallskip
  \item for every $\theta\in\R$, the functions $\beta\mapsto \HV^L(G)  (\theta)$ and  $\beta\mapsto \HV^U(G)  (\theta)$  
  are nondecreasing and Lipschitz continuous with respect to $\beta>0$;\smallskip
\item if $G(0)=0$, then $ \HV^L(G)(0)=\HV^U(G)  (0)=\beta$.
\end{enumerate}
If, in addition, $G$ is convex,
then $\HV^L(G)(\theta)= \HV^U(G)(\theta)=:\HV(G)(\theta)$ for all
$\theta\in\R$, and the function $\HV(G):\R\to\R$ is convex.
\end{prop}

\begin{proof} 
  Recall that 
    $u_\theta$ denotes the unique solution of \eqref{eq general HJ} with
  initial condition $u_\theta(0,x,\omega)=\theta x$.

To prove the first assertion, we temporarily denote by $\HV^U(G)  (\theta,\omega)$, $\HV^L(G)(\theta, \omega)$ the limsup and liminf appearing in \eqref{eq referee}, 
respectively. 
Fix $z\in\R$, $\omega\in\Omega$ and set $w(t,x):=u_\theta(t,x+z,\omega)-z\theta$. Then $w$ satisfies $w(0,x)=\theta x$ and 
\[
 \partial_t w=a(z+x,\omega)\partial^2_{xx} w +G(\partial_x w)+  \beta V(z+x,\omega)  \qquad \hbox{in $(0,+\infty)\times\R$}.
\]
By stationarity of $a$ and $V$ and uniqueness, it follows that
$w=u(\cdot,\cdot,\tau_z\omega)$. Hence
\[
 \HV^U(G)  (\theta,\tau_z\omega)
 =
 \limsup_{t\to +\infty}\ \frac{u_\theta(t,z,\omega)-z\theta}{t}
 =
 \limsup_{t\to +\infty}\ \frac{u_\theta(t,0,\omega)}{t}
 =
 \HV^U(G)  (\theta,\omega),
\]
where for the second equality we have used the fact that
  $u_\theta$ is Lipschitz,
see Proposition \ref{prop Lip estimates}. By ergodicity, we conclude
that the map $\omega\mapsto\HV^U(G) (\theta,\omega)$ is almost surely
constant.  Similar argument applies to
  $\omega\mapsto\HV^L(G) (\theta,\omega)$. 

\indent  (i) The first inequality follows by the very definition of $\HV^L(G)$ and $\HV^U(G)$. To prove the second inequality, set 
$\alpha(h):=\alpha_0|h|^\gamma-1/\alpha_0$ and note that the function 
$v_\theta(t,x):=\theta x+\alpha(|\theta|)t$ is a subsolution of \eqref{eq general HJ} with $v_\theta(0,x)=\theta x$. By applying the comparison principle stated in 
Proposition \ref{prop Lip comp} to the functions $v_\theta(t,x)-\theta x$ and $u_\theta(t,x)-\theta x$ we get 
\[
\HV^L(G)(\theta)
= 
\liminf_{t\to +\infty} \frac{u_\theta(t,0,\omega)}{t}
\geqslant 
\liminf_{t\to +\infty} \frac{v_\theta(t,0,\omega)}{t}
=\alpha(|\theta|).
\]

(ii) The assertion is a direct consequence of Proposition \ref{prop lower bound} below.
 
(iii) We prove the assertion for $\HV^L(G)$ only, {the argument for
$\HV^U(G)$ being analogous}.  Let $\beta_1,\,\beta_2\in (0,+\infty)$
and denote by $u_i$ the solution of \eqref{eq general HJ} with
$\beta=\beta_i$ satisfying $u_i(0,x)=\theta x$ in $\R$. Then
\[
 \partial_{t }u_1\leqslant a(x,\omega) \partial^2_{xx} u_1 +G(\partial_x u_1)+\beta_2 V(x,\omega)+|\beta_1-\beta_2|\quad \hbox{in $\cyl$.}
\]
This means that $u_1-|\beta_1-\beta_2|t$ is a subsolution of \eqref{eq general HJ} with $\beta:=\beta_2$ and initial condition 
$\theta x$. By comparison we infer $u_2\geqslant u_1-|\beta_1-\beta_2|t$, hence
\begin{eqnarray*}
\HVtwo^L(G) (\theta)
 &=&
 \liminf_{t\to +\infty} \frac{u_2(t,0,\omega)}{t}\\
 &\geqslant&
 \liminf_{t\to +\infty} \frac{u_1(t,0,\omega)-|\beta_1-\beta_2|t}{t}
 =
   \HVone^L(G) (\theta)-|\beta_1-\beta_2|.
\end{eqnarray*}
By interchanging the role of $\beta_1$ and $\beta_2$ we infer  $\big| \HVone^L(G) (\theta)- \HVtwo^L(G) (\theta)\big|\leqslant |\beta_1-\beta_2|$.  

If $\beta_1{\ge}\beta_2$, we furthermore have
\[
 \partial_{t }u_1\geqslant a(x,\omega) \partial^2_{xx} u_1 +G(\partial_x u_1) +\beta_2 V(x,\omega),\quad \hbox{in $\cyl$.}
\]
meaning that $u_1$ is a supersolution of \eqref{eq general HJ} with $\beta:=\beta_2$. By comparison we infer $u_2\leqslant u_1$, hence
\[
  \HVtwo^L(G) (\theta)
 =
 \liminf_{t\to +\infty} \frac{u_2(t,0,\omega)}{t}
 \leqslant
 \liminf_{t\to +\infty} \frac{u_1(t,0,\omega)}{t}
 =
 \HVone^L(G) (\theta), 
\]
yielding the claimed monotonicity of $\beta\mapsto\HV^L(G)  (\theta)$. 

(iv) It suffices to show that $ \HV^U(G)   (0)\leqslant \beta$. This follows from the fact that the function $w(t,x)=\beta t$ is a supersolution of \eqref{eq general HJ} satisfying $w(0,x)=0$, as it can be easily seen. By comparison, 
we get $u_0(t,x)\leqslant  \beta t$, yielding $ \HV^U(G) (0)\leqslant \beta$.

The last assertion follows by well known results in stationary ergodic homogenization.
\end{proof}

We now return to the setting of Section \ref{sez intro main}. 
The next proposition shows that without loss of generality we can
assume that $c_+=-c_-$.

\begin{prop}\label{prop reduction}
For given $c_+\geqslant c_-$ in $\R$, let $G^+,G^-:\R\to\R$ be functions satisfying (G1)-(G2) with $G^+(c_+)=G^-(c_-)=0$ and set  
$G_{c_\pm}:=G^-\wedge G^+$. Let 
\[
\tilde G^\pm(p):=G^\pm\left(p+\frac{c_++c_-}{2}\right),
\quad
\tilde G(p):=G_{c_\pm}\left(p+\frac{c_++c_-}{2}\right)=(\tilde G^+\wedge \tilde G^-)(p),
\]
for every $p\in\R$. If  \eqref{eq general HJ} homogenizes with $G:=\tilde G$, 
then the same holds with $G:=G_{c_\pm}$. Furthermore, the associated 
effective Hamiltonians satisfy the following relation:
\begin{equation}\label{claim reduction}
 \HV(G_{c_\pm}) (\theta)= \HV(\tilde G)\left(\theta-\frac{c_++c_-}{2} \right)\qquad\hbox{for all $\theta\in\R$.}
\end{equation} 
\end{prop}

Note that $\tilde G^+(c)=\tilde G^-(-c)=0$ with $c=(c_+ - c_-)/2$.

\begin{proof}
Let us set $k:=-(c_++c_-)/2$. For every fixed $\theta\in\R$, let us denote by $v_{\theta}$ the solution of \eqref{eq general HJ} with  
$G:=G_{c_\pm}$  and initial condition $v_\theta(0,x)=\theta x$. The function $u(t,x)=v_\theta(t,x)+kx$ 
solves equation \eqref{eq general HJ} with  $G:=\tilde G$ and initial condition 
$u(0,x)=\big(\theta + k\big) x$. Since the latter equation homogenizes by hypothesis, we get 
\[
\HV(\tilde G)\left(\theta+k \right)
=
\lim_{t\to +\infty}\ \frac{u(t,0,\omega)}{t}
=
\lim_{t\to +\infty}\ \frac{v_\theta(t,0,\omega)}{t}, 
\]
yielding \eqref{claim reduction}. 
\end{proof}

%
%
%
%
%
%
We shall therefore restrict our attention to the case $c_+=-c_-=: c$
and set $G_c:=G^-\wedge G^+$.  Up to replacing $G_c$ with
$\check G_c(p):=G_c(-p)$, $p\in\R$, we can furthermore assume, without
loss of generality, that $\widehat{p}\geqslant 0$.  Note that
\begin{equation}\label{eq trivial}
G^-(\widehat{p})=G^+(\widehat{p})=G_c(\widehat{p})=\max_{p\in[-c,c]} G_c(p).
\end{equation}


\section{Upper and lower bounds}\label{bounds}

Our goal is to show that $\HV^L(G_c)=\HV^U(G_c)$, which is a necessary
and sufficient condition for homogenization of \eqref{eq general HJ}
with $G:=G_c$, as remarked above. We
start by proving suitable {lower and upper bounds for
  these lower and upper limits.}

\subsection{Lower bound}
We aim at proving the following lower bound:
\begin{equation}\label{eq lower bound 1}
\HV^L(G_c) (\theta)\geqslant  \beta \qquad \hbox{for every  $\theta\in\R$.}
\end{equation}
This follows from the following more general result:

\begin{prop}\label{prop lower bound}
  Let $G\in \Ham$ and $V$ satisfy the scaled hill condition
    (V2)(a),(h). Then
\begin{equation}\label{eq lower bound 1bis}
\HV^L(G)  (\theta)\geqslant  \beta \qquad \hbox{for every  $\theta\in\R$.}
\end{equation}
\end{prop}

\begin{proof}
Let us fix $\theta\in\R$. We want to find a subsolution $v$ to \eqref{eq general HJ} satisfying $v(0,x)\leqslant \theta x$. 
Pick $\eps>0$ and $h\in (0,1)$. Choose $y>0$ big enough so that 
\begin{equation}\label{eq lb 1}
 G(\theta+\eps p)\geqslant \beta h\qquad\hbox{for every $|p|\geq {y}$.}
\end{equation}
Choose $\Omega(h,y)\subseteq \Omega$ of probability 1 as in the scaled hill condition, see (V2). Pick $\omega\in\Omega(h,y)$ and choose $\ell_1<\ell_2$ 
and $\delta$ such that (a),(h) hold in (V2). Pick $x_0\in (\ell_1, \ell_2)$ such that 
\begin{equation}\label{eq lb 2}
 \int_{\ell_1}^{x_0} \frac{1}{a(r,\omega)\vee \delta} \, dr
 =
 \int_{x_0}^{\ell_2} \frac{1}{a(r,\omega)\vee \delta}\, dr
 =y
\end{equation}
and set 
\begin{eqnarray*}
 \cchi(x):=\int_{x_0}^x \frac{1}{a(r,\omega)\vee \delta}\, dr,\qquad && x\in\R\\
 v^\eps(t,x)=\theta x-\eps\int_{x_0}^x \cchi(s)\,ds+(\beta h-\eps)t,\qquad && (t,x)\in\ccyl.  
\end{eqnarray*}
Note that $\partial_x v^\eps(t,x)=\theta -\eps \cchi(x)$,\ \ $a(x,\omega)\partial^2_{xx} v^\eps(t,x)\geq-\eps$\ \ in $\cyl$. 
We are going to show that $v^\eps$ is a subsolution of \eqref{eq general HJ}. Indeed, for every $t>0$ and $x\in\R$ we have
\[
 a(x,\omega) \partial^2_{xx} v^\eps  +G\left(\partial_x v^\eps \right)+\beta V(x,\omega)
 \geqslant
 -\eps+G\left(\theta-\eps \cchi(x)\right)+\beta V(x,\omega)
 \geqslant 
 -\eps+\beta h
\]
For $x\in [\ell_1,\ell_2]$, the above inequality holds true  
for \ $V(\cdot,\omega)\geqslant h$ \ in $[\ell_1,\ell_2]$ by (V2)-(h) and $G\geqslant 0$ in $\R$. 
For $x\in\R\setminus [\ell_1,\ell_2]$, it holds true 
for \ $G\left(\theta-\eps \cchi(x)\right)\geqslant \beta h$ \ in $(-\infty,\ell_1)\cup(\ell_2,+\infty)$ in view of \eqref{eq lb 1}, \eqref{eq lb 2} and 
$V(\cdot,\omega)\geqslant 0$ in $\R$.  
In either case, $v^\eps$ is a subsolution of \eqref{eq general HJ} satisfying
\[
 v^\eps(0,x)=\theta x -\eps\int_{x_0}^x \cchi(s)\,ds\leqslant \theta x.
\]
Let $u_\theta$ be the solution of \eqref{eq general HJ} satisfying
$u_\theta(0,x)=\theta x$.  Since $u_\theta$ is
Lipschitz on $\cyl$, see Proposition \ref{prop Lip estimates}, the
function $u_\theta(t,x,\omega) -\theta x$ is bounded in
$\cTcyl$, for every fixed $T>0$. We can therefore apply the comparison
principle stated in Proposition \ref{prop Lip comp} to
$u_\theta(t,x,\omega) -\theta x$ and $v^\eps(t,x)-\theta x$ with
$G(\theta+\cdot)$ in place of $G$ and get
$u_\theta(t,x,\omega)\geqslant v^\eps(t,x)$ for every $(t,x)\in \cyl$
and $\omega\in\Omega_\eps$. We {conclude that}
\[
 \liminf_{t\to +\infty} \frac{u_\theta(t,0,\omega)}{t}
 \geqslant
 \liminf_{t\to +\infty} \frac{v^\eps(t,0)}{t}=\beta h-{\eps}.
\]
Since this holds for every $\omega\in \Omega(h,y)$ and $\PP(\Omega(h,y))>0$, we infer that 
\[
 \HV^L({G}) (\theta)\geqslant \beta h-{\eps}.
\]
Now let $\eps\to 0^+$ and then $h\to 1^-$ to get the desired lower bound \eqref{eq lower bound 1bis}. 
\end{proof}

\subsection{General upper bound.}
We aim at proving the following general upper bound
\begin{equation}\label{eq general upper bound}
\HV^U(G_c) (\theta)\leqslant \min\left\{\HV(G^-) (\theta),\HV(G^+) (\theta)\right\} \qquad\hbox{for all $\theta\in\R$.}
\end{equation}
Since the function $G^\pm$ are convex,  equation \eqref{eq general HJ} with 
$G:=G^\pm$ homogenizes with effective Hamiltonian $\theta\mapsto\HV(G^\pm) (\theta)$. 
For every fixed $\theta\in\R$, let us denote by $u^\pm_\theta$ the solution to \eqref{eq general HJ} with $G:=G^\pm$ and 
initial condition $u^\pm_\theta(0,x,\omega)=\theta x$. 
Since $G_c\leqslant G^\pm$, by the comparison principle we infer 
\[
 \HV^U(G_c) (\theta)
 =
 \limsup_{t\to +\infty}\ \frac{u_\theta(t,0,\omega)}{t}
 \leqslant
  \limsup_{t\to +\infty}\ \frac{u^\pm_\theta(t,0,\omega)}{t}
  =
 \HV(G^\pm) (\theta),
   \]
yielding the sought general upper bound. \qed

\subsection{Upper bound when $|\theta|\leqslant c$.} We aim at proving the following upper bound
\begin{equation}\label{eq upper bound 1}
\HV^U(G_c) (\theta)\leqslant \max\{\beta,G_c(\widehat{p})\} \qquad \hbox{for  $|\theta|\leqslant c$.}
\end{equation}

This bound follows from the next proposition by recalling \eqref{eq trivial}, i.e.\ that $G_c(\widehat{p})=\max_{p\in[-c,c]} G_c(p)$. 

\begin{prop}\label{prop upper bound}
  Let $G\in \Ham$ be such that $G(\pm c)=0$ and $V$
    satisfy the scaled valley condition (V2)(a),(v). Then
\begin{equation}\label{eq upper bound 1bis}
\HV^U(G)  (\theta)\leqslant \max\{\beta, \max_{[-c,c]}G(\cdot)\}\qquad \hbox{for every  $|\theta|\leqslant c$.}
\end{equation}
\end{prop}

%

\begin{proof}
  Let us {denote the right-hand side of \eqref{eq upper
      bound 1bis} by $\eta$}.  We would like to find a supersolution $w$ to
  \eqref{eq general HJ} of the form $w(t,x)=:\tilde w(x)+\eta t$ with
  $\tilde w(x)\geqslant \theta x$. The {naive} idea is to set
  $\tilde w(x):=c|x|$. An easy computation shows that $\tilde w$
  satisfies, for $x{\ne} 0$,
\[
a(x,\omega)  \partial_x^2 \tilde w+G (\partial_x \tilde w )+\beta V(x,\omega) 
=
G (\pm c)+\beta V(x,\omega) 
\leqslant 
\beta
\leqslant 
\eta,
\]
so $w(t,x)$ is a supersolution to \eqref{eq general HJ} in $\R\setminus\{0\}\times(0,+\infty)$. The problem is that $w(t,x)$ is not a supersolution at $x=0$. Note that
$\tilde w(x)=c|x|\geqslant |\theta| |x|\geqslant \theta x$.  

We need to modify the definition of $w(t,x)=c|x|+t \eta $. We begin by
replacing the function $x\mapsto c|x|$ with a smooth one, whose
derivative is equal to $c$ or $-c$ outside a compact interval.  To
this aim, let us fix $h\in (0,1)$ and $y>0$. Choose
$\Omega(h,y)\subseteq \Omega$ of probability 1 as in the scaled valley
condition (V2). Pick $\omega\in\Omega(h,y)$ and choose
$\ell_1<\ell_2$ and $\delta$ such that (a),(v) hold in (V2). Set
$a_\delta(x):=a(x,\omega)\vee\delta$ and pick
$x_0\in (\ell_1, \ell_2)$ such that
\begin{equation}\label{eq ub 1}
 \int_{\ell_1}^{x_0} \frac{1}{a_{\delta}(r)} \, dr
 =
 \int_{x_0}^{\ell_2} \frac{1}{a_\delta(r)}\, dr
 =y
\end{equation}
We define a function $s:\R\to\R$ by setting 
\begin{eqnarray*}
s(x):=
\begin{cases}
-c & \quad\hbox{if $x\leqslant \ell_1$}\\
\dfrac{c}{2}\left(\dfrac1y\displaystyle
\int_{x_0}^x \dfrac{1}{a_\delta(r)} dr \right)\left( 3-\left(\dfrac{1}{y}\displaystyle
\int_{x_0}^x \dfrac{1}{a_\delta(r)}dr\right)^2 \right) &\quad \hbox{if $x\in (\ell_1,\ell_2)$}\\
c & \quad\hbox{if $x\geqslant \ell_2$.}
\end{cases}
\end{eqnarray*}
First notice that $s(\ell_1^+)=-c$, $s(\ell^-_2)=c$, yielding that $s$ is continuous. Furthermore,  
\[
s'(x)=\dfrac{3c}{2y\,a_\delta(x)} \left(1-\left(\dfrac{1}{y}\int_{x_0}^x\dfrac{1}{a_\delta(r)} dr  \right)^2\right)\qquad\hbox{for $x\in (\ell_1,\ell_2)$},
\]
hence $s'(\ell_1^+)=s'(\ell_2^-)=0$, showing that $s$ is actually of class $C^1$. Also notice that $s'>0$ in $(\ell_1,\ell_2)$, in particular 
\begin{equation}\label{eq ub 2}
-c=s(\ell_1)<s(x)<s(\ell_2)=c\qquad \hbox{for all $x\in (\ell_1,\ell_2)$.}
\end{equation}

For $x\in [\ell_1,\ell_2]$, we get
\begin{equation}\label{eq ub3}
a(x,\omega)s'(x)+G(s(x))+\beta V(x,\omega) 
\leqslant 
\frac{3c}{y}+G(s(x))+\beta V(x,\omega)
\leqslant 
\dfrac{3c}{2y}+\eta+\beta h
\end{equation}
in view of \eqref{eq ub3} and of the fact that $V(\cdot,\omega)\leqslant h$ in $[\ell_1,\ell_2]$ by (V2)-(v). 

For $x\in(-\infty,\ell_1)\cup(\ell_2,+\infty)$, we have 
\begin{equation}\label{eq ub4}
a(x,\omega)s'(x)+G(s(x))+\beta V(x,\omega) \leqslant G(\pm c)+\beta=\beta\leqslant \eta.
\end{equation}
Now set 
\[
w(t,x):=k(\theta)+\int_{x_0}^x s(r)\,dr + \left(\eta+\dfrac{3c}{2y}+\beta h\right) t,\qquad (t,x)\in \ccyl,
\]
with $k(\theta)$ chosen big enough so that $w(0,x)\geqslant \theta x$. For instance, take 
\[
-k(\theta):=\theta x_0+\min_{\ell_1\leqslant x \leqslant \ell_2}\int_{x_0}^x\left(s(r)-\theta\right)\,dr.
\]
From what proved above we infer that $w$ is a $C^2$ (classical) supersolution of \eqref{eq general HJ} satisfying $w(0,x)\geqslant \theta x$. 
Let $u_\theta$ be the solution of \eqref{eq general HJ} satisfying
$u_\theta(0,x)=\theta x$.  Since $u_\theta$ is Lipschitz on $\cyl$ by
Proposition \ref{prop Lip estimates}, the function
$u_\theta(t,x) -\theta x$ is bounded in $\cTcyl$ for every fixed
$T>0$. We can therefore apply the comparison principle stated in
Proposition \ref{prop Lip comp} to $u_\theta(t,x) -\theta x$ and
$w(t,x)-\theta x$ with $G(\theta+\cdot)$ in place of $G$ and get
\[
u_\theta(t,x,\omega)\leqslant w(t,x)\qquad\hbox{for every $(t,x)\in (0,+\infty)\times\R$,}
\]
in particular 
\[
\limsup_{t\to +\infty} \frac{u_\theta(t,0,\omega)}{t}
\leqslant 
\limsup_{t\to +\infty} \frac{w(t,0)}{t}
=
\eta+\dfrac{3c}{2y}+\beta h.
\]
Since this holds for every $\omega\in \Omega(h,y)$ and $\PP(\Omega(h,y))>0$, we infer that 
\[
\HV^U(G)  (\theta)\leqslant \eta+\dfrac{3c}{2y}+\beta h.
\]
Now we send $h\to 0^+$ and $y\to +\infty$ to get the upper bound \eqref{eq upper  bound 1}. 
\end{proof}

\section{Existence of correctors}\label{sez existence of
  correctors}

The goal of the present section is to single out conditions on
$\theta\in\R$ under which we have correctors for \eqref{eq general HJ}.  In the sequel, we will say that
a function $u:\R\to\R$ is {\em sublinear} or has {\em sublinear
  growth} to mean that
\[
\lim_{|x|\to +\infty} \frac{u(x)}{1+|x|}=0. 
\]

\subsection{Correctors}\label{subsect correctors}
In this subsection, we collect and prove some key results we shall
need for our analysis.  We shall assume that
  $G:\R\to [0,+\infty)$ is a function in $\Ham$ satisfying the
following additional assumption:
\begin{itemize}
\item[(G3)] \quad $G(0)=0$;\smallskip
\item[(G4)] \quad $G$ is convex.
\end{itemize}
Notice that conditions (G3)-(G4) and the fact that $G\geqslant 0$ in
$\R$ imply that $G$ is nonincreasing  in
$(-\infty,0]$ and nondecreasing  in
$[0,+\infty)$.  By known results in stationary ergodic homogenization,
the equation \eqref{eq general HJ} homogenizes. We shall denote by
$\HV(G) $ the corresponding effective Hamiltonian.  Since
$V\geqslant 0$, we get
\begin{equation}\label{eq HV(G)>G}
\HV(G)(\theta)  \geqslant G(\theta)\qquad\hbox{for all $\theta\in\R$}.
\end{equation}
We know that $\HV(G)  $ is convex and coercive and has a minimum at $0$ with $\HV(G)  (0)=\beta$, see Proposition \ref{prop Lambda}.
The following proposition shows the
  existence of a Lipschitz continuous corrector with stationary
  gradient for every $\theta$ satisfying $\HV(G)(\theta)>\beta$.

\begin{prop}\label{prop CS}
Let $\theta\in\R$ {be} such that $\HV(G)  (\theta)>\beta$. Then there exists a random variable 
$\Omega\ni \omega\mapsto F_\theta(\cdot,\omega)\in\D{C}(\R)$ such that, for every $\omega$ in a set $\Omega_{\theta}$ of probability 1, 
$F_\theta(\cdot,\omega)$ is the unique sublinear viscosity solution of the stationary viscous Hamilton--Jacobi equation
\begin{equation}\label{eq corrector}
a(x,\omega) u'' +G(\theta+u ')+\beta V(x,\omega)=\HV(G)  (\theta)\qquad\hbox{in $\R$}
\end{equation}
satisfying  $F_\theta(0,\omega)=0$ for every $\omega\in \Omega$. The set $\Omega_{\theta}$ is invariant under the action of $(\tau_z)_{z\in\R}$, i.e. 
$\tau_z\big(\Omega_{\theta})=\Omega_{\theta}$ for every $z\in\R$. 
Furthermore, the function $F_\theta(\cdot,\omega)$ is $\kappa(\theta)$--Lipschitz continuous on $\R$ for  $\PP$--a.e. $\omega\in\Omega$, where 
$\kappa:\R\to [0,+\infty)$ is a locally bounded function, and  has stationary gradient, i.e. for every $\omega$ in a 
set of probability 1 we have
\[
F_\theta'(\cdot+z,\omega)=F_\theta'(\cdot,\tau_z\omega)\qquad\hbox{a.e. on $\R$ \quad for every $z\in\R$.}
\]
\end{prop}

\begin{proof}
Let us set  $\check G(p):=G(-p)$ for all $p\in\R$. Then equation \eqref{eq general HJ} with $\check G$ in place of $G$ also homogenizes, with effective 
Hamiltonian $\HV(\check G)$ satisfying $\HV(\check G)  (-\theta)=\HV(G)  (\theta)$. 
The function \ $u (x,\omega):=-F_\theta(x,\omega)$\ \ is a viscosity solution to
\begin{equation}\label{eq CS}
-a(x,\omega) u'' +\check G(-\theta+ u ')+\beta V(x,\omega)=\HV(\check G)(-\theta)\qquad\hbox{in $\R$.}
\end{equation}
Hence, it will be enough to prove the assertion for $u$. 
We want to apply Theorem 2.1 in \cite{CaSo17}, which was proved 
under the following\smallskip\\ 
%
%
%
%
{\bf Assumption (H):}  {\em for any $\theta\in\R$, the approximate corrector equation
\begin{equation}\label{eq approximate corrector}
\lambda v_{\lambda,\theta}-a(x,\omega) v_{\lambda,\theta}''+\check G(\theta+ v_{\lambda,\theta}')+\beta V(x,\omega)=0  \qquad\hbox{in $\R$}
\end{equation}
satisfies a comparison principle in $\D{C}_b(\R)$, and, for any $R>0$, there exists a constant $\kappa(R)>0$ such that, if $|\theta|\leqslant R$, 
then the unique bounded solution $v_{\lambda,\theta}$ of \eqref{eq approximate corrector} satisfies} 
\begin{equation}\label{condition H}
\|\lambda v_{\lambda,\theta}\|_\infty+\|v^{'}_{\lambda,\theta}\|_\infty \leqslant \kappa(R)\qquad\hbox{for all $\lambda>0$}. 
\end{equation}

Let us check that assumption (H) holds in our framework. The validity of the required comparison principle for \eqref{eq approximate corrector} is 
guaranteed by \cite[Theorem 2.1]{AT}. Since the functions $\pm C(R)/\lambda$ with 
$C(R):=\beta+\sup_{|\theta|\leqslant R} \check G(\theta)$ are a bounded super- and sub- solution to \eqref{eq approximate corrector}, respectively, we immediately derive by comparison 
that \ $\| \lambda v_{\lambda,\theta}\|_\infty\leqslant C(R)$. This bound, together with the quantitative Lipschitz bounds for $v_{\lambda,\theta}$  provided by 
\cite[Theorem 3.1]{AT}, imply that \eqref{condition H} holds for a suitable nondecreasing function $\kappa:\R\to [0,+\infty)$.

Following \cite{CaSo17}, we  choose $R>|\theta|$ and denote by 
$$
\Theta:=\left\{v\in\Lip(\R)\,\mid\,v(0)=0, \|v'\|_\infty\leqslant \kappa(R)\,\right\}
$$
the metric subspace of $\CC(\R)$. It is easily seen that $\Theta$ is a compact metric space.  
The inequality  $\HV(\check G)  (-\theta)>\beta$ implies $\theta \ne 0$, so, according to Corollary \ref{cor unique solution}, for each fixed 
$\omega\in\Omega$ there is at most one sublinear solution of \eqref{eq CS} in $\Theta$, let us call it $\hat u(\cdot,\omega)$. 
Now note that $-\theta$ is an extremal point of the closed interval $\{\tilde\theta\in\R\,\mid\,\HV(\check G) (\tilde \theta)\leqslant \HV(\check G) (-\theta)\}$, 
for $\HV(\check G) (-\theta)>\beta=\min \HV(\check G)$ and $\HV(\check G)$ is convex. 
In \cite[Theorem 2.1]{CaSo17} the authors have obtained a probability measure $\mu$ on $\Omega\times\Theta$ (we {can} forget about the third coordinate in $\tilde\Omega$ as, in our setting, 
the restriction of $\mu$ on the third coordinate is a Dirac mass at  $\HV(\check G)  (-\theta)$) 
such that $\mu(E_\theta )=1$, where 
\[
E_\theta :=\left\{ (\omega,v)\in\Omega\times\Theta\,\mid\,\hbox{$v$ is a sublinear solution of \eqref{eq CS}}\right\}.
\]
Furthermore, the set $E_\theta $ is invariant under the shifts
$\tilde\tau_z:\ (\omega,v)\mapsto (\tau_z \omega, v(\cdot+z)-v(z))$.
Indeed, if $v\in\Theta$ is a sublinear solution of \eqref{eq CS} for
some $\omega$, then $v(\cdot+z)-v(z)$ belongs to $\Theta$ and is a sublinear solution of
\eqref{eq CS} with $\tau_z\omega$ in place of $\omega$, since
$V(\cdot+z,\omega)=V(\cdot,\tau_z\omega)$ in $\R$. In particular, we
get that $(\omega,v)\in E_\theta $ implies $v=\hat u(\cdot,\omega)$.
Let $\Omega_{\theta}:=\pi_1(E_\theta )$, where
$\pi_1:\Omega\times\Theta\to\Omega$ denotes the standard projection,
and recall that the first marginal of the measure $\mu$ is $\PP$. Then
$\Omega_{\theta}\in\F$ and
$\tau_z(\Omega_{\theta})=\Omega_{\theta}$ for all $z\in\R$, in the
light of what previously remarked.

By making use of the disintegration theorem (see \cite[Theorem
10.2.2]{Dud}) we get that there exists {a family
  of random} probability measures $\mu_\omega$ on $\Theta$ such that
$\mu=\mu_\omega\otimes\PP$, i.e.
\[
\int_{\Omega\times\Theta} \phi(\omega,v)\,d\mu(\omega,v)=\int_\Omega\left( \int_\Theta \phi(\omega,v)\,d\mu_\omega(v)\right)\,d\PP(\omega)
\quad
\hbox{for all $\phi\in\CC(\Omega\times\Theta)$.}
\]
By what observed above, for every $\omega\in\Omega_{\theta}$ the
measure $\mu_\omega$ is the Dirac measure concentrated at
$\hat u(\cdot,\omega)$, hence the map
$\Omega_{\theta}\ni \omega\mapsto \hat u(\cdot,\omega)\in\Theta$ is a
random variable.  The sought random variable $u:\Omega\mapsto\CC(\R)$
is thus obtained by setting
\[
u(\cdot,\omega)=\hat u(\cdot,\omega)\quad\hbox{if $\omega\in\Omega_{\theta}$,}\qquad u(\cdot,\omega)=0\quad\hbox{otherwise.}
\] Lastly, for every $\omega\in\Omega_{\theta}$ and $z\in\R$, we have
$u(\cdot+z,\omega)-u(\cdot,\omega)=u(\cdot,\tau_z\omega)$ in $\R$ in
view of Corollary \ref{cor unique solution}, since both are sublinear
solutions of \eqref{eq CS} with $\tau_z\omega$ in place of
$\omega$. By differentiating this identity we get
$u'(\cdot+z,\omega)=u'(\cdot,\tau_z\omega)$ a.e. in $\R$, for every $z\in\R$
and $\omega\in\Omega_{\theta}$.
\end{proof}

From now on, when we say that a random variable
$\Omega\ni \omega\mapsto F_\theta(\cdot,\omega)\in\D{C}(\R)$ is a {\em
  corrector} for \eqref{eq corrector} we will mean that
$F_\theta(\cdot,\omega)$ is a sublinear, Lipschitz continuous
viscosity solution of \eqref{eq corrector} satisfying
$F_\theta(0,\omega)=0$ for every $\omega\in\Omega_{\theta}$, where
$\Omega_{\theta}$ is a set of probability 1 which is invariant under
the action of $(\tau_z)_{z\in\R}$, with no further specification. In
view of what remarked above, a corrector automatically possesses
stationary gradient. We point out that our arguments below do not use this property.

We are interested in obtaining suitable upper and lower bounds for
$F'_\theta$ depending on $\theta$. We start with the following lemma.

\begin{lemma}\label{lemma linear subsolution}
Let us consider the following viscous Hamilton--Jacobi equation
\begin{equation}\label{eq linear subsolution}
 -a(x,\omega) u'' +\check G(u')+\beta V(x,\omega)=\lambda\qquad\hbox{in $I$,}
\end{equation}
where $\lambda>\beta$ and $I$ is either $(-\infty,y)$ or $(y,+\infty)$ for a fixed $y\in\R$. 
\begin{enumerate}[(i)]
 \item Let $I=(-\infty,y)$ and $a^-_\lambda,b^-_\lambda>0$ such that $G(a^-_\lambda)=\lambda-\beta$, $G(b^{-}_\lambda)=\lambda$. 
 Then the functions 
 \[
  v_-(x):=a_\lambda^- |x-y|=-a_\lambda^- (x-y),
  \qquad
  w_-(x):=b_\lambda^- |x-y|=-b_\lambda^- (x-y)
 \]
are, respectively, a sub- and a super- solution of \eqref{eq linear subsolution} in $I=(-\infty,y)$.\smallskip
\item Let $I=(y,+\infty)$ and $a^+_\lambda, b^+_\lambda>0$ such that $G(-a^+_\lambda)=\lambda-\beta$, $G(-b^{+}_\lambda)=\lambda$. 
 Then the functions 
 \[
  v_+(x):=a_\lambda^+ |x-y|=a_\lambda^+ (x-y),
  \qquad
  w_+(x):=b_\lambda^+ |x-y|=b_\lambda^+ (x-y)
 \]
are, respectively, a sub- and a super- solution of \eqref{eq linear subsolution} in $I=(y,+\infty)$. 
\end{enumerate}
\end{lemma}
\begin{proof}
Let us prove (i). We have
\[
 -a(x,\omega) (v_-)''(x)+\check G (v_-'(x))+\beta V(x,\omega)
 \leqslant
 \check G(-a^-_\lambda)+\beta
 =\lambda-\beta+\beta=\lambda\quad\hbox{for all $x<y$,} 
\]
showing that $v_-$ is a subsolution of \eqref{eq linear subsolution} in $I=(-\infty,y)$. Analogously,
\[
 -a(x,\omega) (w_-)''(x)+\check G (w_-'(x))+\beta V(x,\omega)
 \geqslant
 \check G(-b^-_\lambda)
 =\lambda\qquad\hbox{for all $x<y$,} 
\]
showing that $w_-$ is a supersolution of \eqref{eq linear subsolution} in $I=(-\infty,y)$. 
The proof of (ii) is similar and is omitted.
\end{proof}

By comparison,  we get the following statement.

\begin{prop}\label{prop inequality corrector}
Let $\theta\in\R$ such that $\HV(G)  (\theta)>\beta$. Set $\lambda:=\HV(G)  (\theta)$. 
For every $y\in\R$ and $\omega\in\Omega_{\theta}$, the following holds:
\begin{enumerate}[(i)]
 \item if $\theta>0$, then 
 \[
  a^-_\lambda (x-y) \geqslant \theta (x-y)+ F_\theta(x,\omega) - F_\theta(y,\omega) \geqslant b^-_\lambda (x-y)\qquad\hbox{for all $x\in (-\infty, y)$,}
 \]
with $b^-_\lambda>a^-_\lambda>0$ such that $G(a^-_\lambda)=\lambda-\beta$, $G(b^{-}_\lambda)=\lambda$;\smallskip
\item if $\theta<0$, then 
 \[
  -a^+_\lambda (x-y) \geqslant \theta (x-y)+ F_\theta(x,\omega) - F_\theta(y,\omega)\geqslant -b^+_\lambda (x-y)\qquad\hbox{for all $x\in (y,+\infty)$,}
 \]
with $b^+_\lambda>a^+_\lambda>0$ such that $G(-a^+_\lambda)=\lambda-\beta$, $G(-b^{+}_\lambda)=\lambda$.
\end{enumerate}
\end{prop}

\begin{proof}
By Proposition \ref{prop CS}, we have that the function 
\[
u(x):=-\left(\theta x+F_\theta(x,\omega)\right)+\theta y+ F_{\theta}(y,\omega)
\] 
is a Lipschitz continuous solution to \eqref{eq linear subsolution} with $I:=\R$  
satisfying $u(y)=0$.

Let us first  consider the case $\theta>0$. By sublinearity of $F_\theta$, the function $u$ is bounded from below in $I=(-\infty,y)$. 
By Theorem \ref{teo comparison bis} and Lemma \ref{lemma linear subsolution} we have 
\[
-a_\lambda^-(x-y)=v_-(x)\leqslant u(x)=-\big(\theta x+ F_\theta(x,\omega)\big)+\theta y+ F_{\theta}(y,\omega)
\qquad\hbox{for all $x<y$,}
\]
proving the first inequality of assertion (i). To prove the second one, note that the functions 
$\tilde u(x):=u(x)+\theta x$ and 
$\tilde w(x):=w_-(x)+\theta x=-b^-_\lambda(x-y)+\theta x$ are, 
respectively, a sub- and a super- solutions of 
\[
 -a(x,\omega) u'' +\check G(-\theta+u')+\beta V(x,\omega)=\lambda \qquad \hbox{in $I=(-\infty,y).$}
\]
Furthermore, $G(b^-_\lambda)=\lambda=\HV(G)  (\theta)\geqslant G(\theta)$ in view of \eqref{eq HV(G)>G}, so $b^-_\lambda\geqslant \theta >0$ by 
monotonicity of $G$ on $[0,+\infty)$. 
Then the sub- and super- solution $\tilde u$ and $\tilde w$ satisfy the assumption of Theorem \ref{teo comparison}, which gives 
\[
 -F_\theta(x,\omega)+\theta y+ F_{\theta}(y,\omega)=\tilde u(x)\leqslant \tilde w(x)= -b^-_\lambda(x-y)+\theta x\qquad\hbox{for all $x<y$,}
\]
yielding the second inequality in assertion (i). 

Let us now consider the case $\theta<0$. By sublinearity of $F_\theta$, the function $u$ is bounded from below in $I=(y,+\infty)$. 
By Theorem \ref{teo comparison bis} and Lemma \ref{lemma linear subsolution} we have 
\[
a_\lambda^+(x-y)=v_+(x)\leqslant u(x)=-\big(\theta x+ F_\theta(x,\omega)\big)+\theta y+ F_{\theta}(y,\omega)\qquad\hbox{for all $x>y$,}
\]
proving the first inequality of assertion (ii). To prove the second one, we argue as above with 
$\tilde u(x):=u(x)+\theta x$ and $\tilde w(x):=w_+(x)+\theta x$ for $x\in I=(y,+\infty)$.  
Analogously, we have $G(-b^+_\lambda)=\lambda=\HV(G)  (\theta)\geqslant G(\theta)$, so $-b^+_\lambda\leqslant \theta <0$, i.e. 
$b^+_\lambda+\theta\geqslant 0$. Again, via a direct application of Theorem \ref{teo comparison} we get 
\[
 -F_\theta(x,\omega)+\theta y+ F_{\theta}(y,\omega)=\tilde u(x)\leqslant \tilde w(x)= b^+_\lambda(x-y)+\theta x\qquad\hbox{for all $x>y$,}
\]
yielding the second inequality in assertion (ii). 
\end{proof}

From the previous proposition we infer the following result.

\begin{prop}\label{prop gradient bounds corrector}
Let $\theta\in\R$ such that $\HV(G)  (\theta)>\beta$. Set $\lambda:=\HV(G)  (\theta)$. For every $\omega\in\Omega_{\theta}$, the following holds:
\begin{enumerate}[(i)]
 \item if $\theta>0$, then 
 \[
 a^-_\lambda \leqslant \theta +F'_\theta(y,\omega) \leqslant b^-_\lambda\qquad \hbox{for a.e. $y\in\R$,}
 \]
 with $b^-_\lambda>a^-_\lambda>0$ such that $G(a^-_\lambda)=\lambda-\beta$, $G(b^{-}_\lambda)=\lambda$;\medskip
\item if $\theta<0$, then 
 \[
 -b^+_\lambda \leqslant \theta +F'_\theta(y,\omega) \leqslant -a^+_\lambda\qquad \hbox{for a.e. $y\in\R$,}
 \]
 with $b^+_\lambda>a^+_\lambda>0$ such that $G(-a^+_\lambda)=\lambda-\beta$, $G(-b^{+}_\lambda)=\lambda$.
\end{enumerate}
\end{prop}

\begin{proof}
  Let $y$ be a differentiability point of $F_\theta(\cdot,\omega)$. If
  $\theta>0$, then from Proposition \ref{prop inequality
    corrector}--(i) we get
\[
a_\lambda^- 
\leqslant 
\lim_{h\to 0^-}\frac{\theta h+ F_\theta(y+h,\omega) - F_\theta(y,\omega) }{h}
\leqslant 
b_\lambda^-
\]
yielding assertion (i). If $\theta<0$, we make use of Proposition \ref{prop inequality corrector}--(ii) and get
\[
-a_\lambda^+ 
\geqslant 
\lim_{h\to 0^+}\frac{\theta h+ F_\theta(y+h,\omega) - F_\theta(y,\omega) }{h}
\geqslant 
-b_\lambda^+, 
\]
yielding assertion (ii). 
\end{proof}

\subsection{Outside the flat part}
In this subsection, we  shall prove the following theorem.
\begin{teorema}\label{teo homogenization via correctors}
\ \smallskip\\
\noindent{\em (a)} Assume either one of the following conditions:
\begin{enumerate}[(i)]
\item \quad $\theta<-c$\quad and\quad $\HV(G^-) (\theta)>\beta$;\smallskip
\item \quad $-c<\theta\leqslant\widehat{p}$\quad and\quad $\beta<\HV(G^-) (\theta)\leqslant G_c(\widehat{p})$.\smallskip
\end{enumerate}
\noindent Then
\[
 \HV^L(G_c) (\theta)=\HV^U(G_c) (\theta)=\HV(G^-) (\theta)=\min\{\HV(G^-) (\theta),\HV(G^+) (\theta) \}. 
\]
\noindent{\em (b)} Assume either one of the following conditions:
\begin{enumerate}[(i)]
\item \quad $\theta>c$\quad and\quad $\HV(G^+) (\theta)>\beta$;\smallskip
\item \quad $\widehat{p}\leqslant \theta<c$\quad and\quad $\beta<\HV(G^+) (\theta)\leqslant G_c(\widehat{p})$.
\end{enumerate}
\noindent Then
\[
 \HV^L(G_c) (\theta)=\HV^U(G_c) (\theta)=\HV(G^+) (\theta)=\min\{\HV(G^-) (\theta),\HV(G^+) (\theta)\}. 
\]
\end{teorema}

The proof of this result is based on a series of lemmas, which we shall prove first. 

\begin{lemma}\label{lemma corrector}
Let $F_\theta:\R\times\Omega\to\R$ be a corrector of the equation
\begin{equation}\label{eq lambda corrector}
 a(x,\omega) u'' +G_c(\theta+u')+\beta V(x,\omega)=\lambda\qquad\hbox{in $\R$}
\end{equation}
for some $\theta\in\R$ and $\lambda\in\R$. Then \ \ $\HV^L(G_c) (\theta)=\HV^U(G_c) (\theta)=\lambda$.
\end{lemma}

\begin{proof}
According to Proposition \ref{prop CS}, we know that $F_{\theta}$ is globally Lipschitz on $\R$. Then the function 
\[
 v(t,x):=\theta x + F_{\theta}(x,\omega)+\lambda\,t
\]
is a solution to \eqref{eq general HJ} with $G:=G_c$ and 
initial condition $v(0,x)=F_{\theta}(x,\omega)+\theta x$. Fix $\eps>0$ and choose a constant $k_\eps>0$ large enough so that the function
\[
 v^\eps(t,x)=v(t,x)-\eps\langle x\rangle-k_\eps\qquad \hbox{where}\quad \langle x \rangle :=\sqrt{1+|x|^2}
\]
satisfies $v^\eps(0,x)=F_{\theta}(x,\omega)+\theta x-\eps\langle x\rangle-k_\eps\leqslant \theta x$\quad in $\R$. This is possible since the function 
$F_{\theta}$ has sublinear growth. Now
\[
 \partial_x v^\eps(t,x)=F_{\theta}'+\theta-\eps\frac{x}{\langle x \rangle},\qquad 
 \partial^2_{xx} v^\eps(t,x)=F_{\theta}''-\frac{\eps}{\langle x \rangle^3}
\]
and 
\begin{eqnarray*}
a(x,\omega) \left(\partial^2_x v^\eps\right)&+& G_c\left( \partial_x v^\eps\right)+\beta V(x,\omega)\\
&=&
\big(-\frac{\eps}{\langle x \rangle^3}+ F''_{\theta }\big)a(x,\omega) +G_c\left(\theta+F'_{\theta}-\frac{\eps x}{\langle x \rangle}\right)+\beta V(x,\omega)=:A.
\end{eqnarray*}
From the fact that $|F'_{\theta}|$ is bounded on $\R$ we infer that there exists a constant $C(\theta)$ such that
\[
 A
 \geqslant 
 -C(\theta)\eps+a(x,\omega)  F''+G_c(\theta+F'_{\theta})+\beta V(x,\omega)
 =
 -C(\theta)\eps+\lambda.
\]
This means that the function $\tilde v^\eps(t,x)=v^\eps(t,x)-C(\theta)\eps t$ is a subsolution of \eqref{eq general HJ} with $\tilde v^\eps(0,x)\leqslant \theta x$, hence by comparison
we infer 
\[
 u_\theta(t,x)\geqslant \tilde v^\eps(t,x)\qquad\hbox{for all $(t,x)\in\ccyl$.}
\]
So 
\[
\HV^L(G_c) (\theta)
=
\liminf_{t\to +\infty}\ \frac{u_\theta(t,0,\omega)}{t}
\geqslant
\liminf_{t\to +\infty}\ \frac{\tilde v^\eps(t,0,\omega)}{t}
=
\lambda-C(\theta)\eps.
\]
By letting $\eps\to 0^+$ we obtain the lower bound $\HV^L(G_c) (\theta)\geqslant \lambda$. 
A similar argument gives the upper bound $\HV^U(G_c) (\theta)\leqslant \lambda$, thus proving the assertion.   
\end{proof}


\begin{lemma}\label{lemma corrector G^-}
Let $F^-_{\theta}:\R\times\Omega\to\R$ be a corrector of the equation
\begin{equation}\label{eq corrector G^-}
a(x,\omega) u'' +G^-(\theta+u')+\beta V(x,\omega)=\HV(G^-) (\theta)\qquad\hbox{in $\R$.}
\end{equation}
Assume either one of the following conditions:
\begin{enumerate}[(i)]
\item \quad $\theta<-c$\quad and\quad $\HV(G^-) (\theta)>\beta$;\smallskip
\item \quad $\theta > -c$\quad and\quad $\beta<\HV(G^-) (\theta)\leqslant G^-(\widehat{p})$.
\end{enumerate}
Then for $\PP$--a.e. $\omega\in\Omega$ we have 
\[
\theta+(F^-_{\theta})'(x,\omega) \leqslant \widehat{p}\quad \hbox{for a.e. $x\in\R$}. 
\]
\end{lemma}

\begin{oss}
Note that the inequality $\HV(G^-)(\theta)\leqslant G^-(\widehat{p})$ for $\theta>-c$ implies $\theta\leqslant \widehat{p}$. This follows from the fact that  
$\HV(G^-)\geqslant G^-$ on $\R$ and $G^-$ is nondecreasing on $[-c,+\infty)$.   
\end{oss}

\begin{proof}
We will make use of Proposition \ref{prop gradient bounds corrector} with $G(\cdot):=G^-(\cdot-c)$, $\HV(G)(\cdot)  :=\HV(G^-)(\cdot-c)$ and 
$\theta+c$ in place of $\theta$. Consequently, we will have  
$\lambda:=\HV(G^-) (\theta)$ and $F_{\theta+c}=F^-_\theta$.

\indent (i) The inequality $\theta+c<0$ means \quad  $\theta+c +F'_{\theta+c}(x,\omega) \leqslant -a^+_\lambda$ {for a.e. $x\in\R$} \quad with $a^+_\lambda>0$ 
such that $G^-(-a^+_\lambda-c)=\lambda-\beta>0$, so 
\[
\theta +F'_{\theta+c}(x,\omega) \leqslant -a^+_\lambda-c<0\leqslant \widehat{p}\qquad\hbox{for a.e. $x\in\R$}.
\]
\indent (ii) The inequality $\theta+c>0$ means  \quad $\theta+c +F'_{\theta+c}(x,\omega) \leqslant b^-_\lambda$ {for a.e. $x\in\R$} with 
$b^-_\lambda>0$ such that $G^-(b^{-}_\lambda-c)=\lambda$. Now  \ $G^-(\widehat{p})\geqslant \lambda=G^-(b^{-}_\lambda-c)$, 
 \  so $b^{-}_\lambda\leqslant \widehat{p}+c$ by monotonicity of $G^-$ on $[-c,+\infty)$, yielding 
\[
\theta+F'_{\theta+c}(x,\omega) \leqslant \widehat{p}\qquad\hbox{for a.e. $x\in\R$}.
\]
\end{proof}


\begin{lemma}\label{lemma corrector G^+}
Let $F^+_{\theta }:\R\times\Omega\to\R$ be a corrector of the equation
\begin{equation}\label{eq corrector G^+}
a(x,\omega) u'' +G^+(\theta+u')+\beta V(x,\omega)=\HV(G^+) (\theta)\qquad\hbox{in $\R$.}
\end{equation}
Assume either one of the following conditions:
\begin{enumerate}[(i)]
\item \quad $\theta>c$\quad and\quad $\HV(G^+) (\theta)>\beta$;\smallskip
\item \quad $\theta<c$\quad and\quad $\beta<\HV(G^+) (\theta)\leqslant G^+(\widehat{p})$.
\end{enumerate}
Then for $\PP$--a.e. $\omega\in\Omega$ we have 
\[
\theta+(F^+_{\theta })'(x,\omega) \geqslant \widehat{p}\quad \hbox{for a.e. $x\in\R$}. 
\]
\end{lemma}

\begin{oss}
Note that the inequality $\HV(G^+)(\theta )\leqslant G^+(\widehat{p})$ for $\theta<c$ implies $\theta\geqslant \widehat{p}$. This follows from the fact that  
$\HV(G^+)\geqslant G^+$ on $\R$ and $G^+$ is nonincreasing on $(-\infty,c]$.   
\end{oss}

\begin{proof}
We will make use of Proposition \ref{prop gradient bounds corrector} with $G(\cdot):=G^+(\cdot+c)$, $\HV(G)(\cdot):=\HV(G^+)(\cdot+c) $ and $\theta-c$ 
in place of $\theta$. Consequently, we will have $\lambda:=\HV(G^+) (\theta)$ and $F_{\theta-c}=F^+_\theta$. 

\indent (i) The inequality $\theta-c >0$ means  \quad $\theta-c +F'_{\theta-c}(x,\omega) \geqslant a^-_\lambda$ {for a.e. $x\in\R$} with 
$a^-_\lambda>0$ such that $G^+(a^{-}_\lambda+c)=\lambda-\beta>0$, so
\[
\theta +F'_{\theta-c}(x,\omega) \geqslant c+a^-_\lambda >c\geqslant \widehat{p}\qquad\hbox{for a.e. $x\in\R$}.
\]
\indent (ii) The inequality $\theta-c<0$ means \quad  $\theta-c +F'_{\theta-c}(x,\omega) \geqslant -b^+_\lambda$ {for a.e. $x\in\R$} with $b^+_\lambda>0$ 
such that $G^+(-b^+_\lambda+c)=\lambda>0$. Now \ \ $G^+(\widehat{p})\geqslant \lambda=G^+(-b^{+}_\lambda+c)$, 
\ \  so $0>-b^{+}_\lambda+c\geqslant \widehat{p}$ by monotonicity of $G^+$ on $(-\infty,c]$, yielding 
\[
\theta+F'_{\theta-c}(x,\omega) \geqslant \widehat{p}\qquad\hbox{for a.e. $x\in\R$}.
\]
\end{proof}

We are now ready to prove Theorem \ref{teo homogenization via correctors}.

\begin{proof}[Proof of Theorem \ref{teo homogenization via correctors}]
(a) Let $F^-_{\theta}$ be a corrector of equation \eqref{eq corrector G^-}. According to Lemma \ref{lemma corrector G^-}, 
\[
\theta+(F^-_{\theta})' \leqslant \widehat{p}\qquad\hbox{for a.e. $x\in\R$.}
\]
This implies that any $C^2$ sub or supertangent $\varphi$ to $F^-_{\theta}$ at some $x_0\in\R$ will satisfy \ \ $\theta+\varphi'(x_0)\leqslant \widehat{p}$, 
hence $G^-(\theta+\varphi'(x_0))=G_c(\theta+\varphi'(x_0))$. We derive from this that $F^-_{\theta}$ is a corrector of equation \eqref{eq lambda corrector} with 
$\lambda:=\HV(G^-) (\theta)$. In view of Lemma \ref{lemma corrector} and of the upper bound \eqref{eq general upper bound}, we get the assertion. 
The proof of item (b) is similar. 
\end{proof}

\section{Proof of Theorem \ref{teo main 1}}\label{sez teo main 1}


In this section we prove Theorem \ref{teo main 1}.
When $\theta\leqslant -c$, we have 
\[
  \HV^L(G_c)(\theta)= \HV^U(G_c)(\theta)=\HV(G^-) (\theta).
\]
This is a direct consequence of Theorem \ref{teo homogenization via correctors} when $\HV(G^-) (\theta )>\beta$, but it is also 
true when $\HV(G^-) (\theta )=\beta$ in view of the lower bound \eqref{eq lower bound 1} and the general upper bound 
\eqref{eq general upper bound}. 
When $\theta\geqslant c$, a similar argument yields 
\[
  \HV^L(G_c)(\theta)= \HV^U(G_c)(\theta)=\HV(G^+) (\theta ).
\]
When $\theta \in (-c,c)$, we have to proceed differently according to whether 
$\beta\geqslant G_c(\widehat{p})$ or $\beta< G_c(\widehat{p})$. 

\subsection{The case $\beta\geqslant G_c(\widehat{p})$}

When $|\theta|< c$, the lower bound \eqref{eq lower bound 1} matches with the upper bound \eqref{eq upper bound 1}, hence we get 
\[
 \HV^L(G_c)(\theta)= \HV^U(G_c)(\theta)=\beta. 
\]
We have thus shown that 
\begin{equation*}  
\HV (G_c)(\theta)=
  \begin{cases}
    \HV (G^+)(\theta)&\text{if}\quad \theta> c\\
    \beta &\text{if }\quad -c\leqslant \theta\leqslant c\\ 
    \HV (G^-)(\theta)&\text{if}\quad \theta< -c.
  \end{cases}
\end{equation*}
In other words, $\HV (G_c)(\theta)$ is the (lower) convex envelope
of $\HV (G^-)(\theta)$ and $\HV (G^+)(\theta)$.
\subsection{The case $\beta< G_c(\widehat{p})$} 
Let  $\theta\in (-c,\widehat{p}]$. 
If  
$\beta<\HV(G^-) (\theta )\leqslant G_c(\widehat{p})$, Theorem \ref{teo homogenization via correctors} yields  
\begin{equation}\label{strong control step 1}
\HV^L(G_c)(\theta)= \HV^U(G_c)(\theta)=\HV(G^-) (\theta ).
\end{equation}
Let us now consider the case $\HV(G^-) (\theta )> G_c(\widehat{p})$. We first remark that 
\[
\HVzero(G^-)(\theta )=G^-(\theta )=G_c(\theta)\leqslant G_c(\widehat{p}).
\]
By Proposition \ref{prop Lambda}, we know that the map 
$\tilde\beta\mapsto \HVtilde^L(G)(\theta )$  is continuous and nondecreasing on $[0,+\infty)$ with  
$G:=G^-$ or $G:=G_c$.  
We infer that there exists a $\beta_-\in [0,\beta)$ such that 
$\HVminus(G^-)(\theta )=G_c(\widehat{p})$, hence by the previous step we get that 
\eqref{strong control step 1} holds with $\beta_-$ in place of $\beta$. By monotonicity  
we get
\[
 \HV^L(G_c)(\theta) \geqslant \HVminus^L(G_c)(\theta)=\HVminus(G^-)(\theta )=G_c(\widehat{p}).
\]
By taking into account the upper bound \eqref{eq upper bound 1}, we conclude that 
\[
\HV^U(G_c)(\theta)=\HV^L(G_c)(\theta)=G_c(\widehat{p}).
\]
When $\theta \in [\widehat{p}, c)$, arguing analogously we get 
\begin{eqnarray*}
\HV^U(G_c)(\theta)=\HV^L(G_c)(\theta)=
\begin{cases}
\HV(G^+) (\theta ) & \ \hbox{if \quad $\beta<\HV(G^+) (\theta )\leqslant G_c(\widehat{p})$}\\
G_c(\widehat{p}) & \ \hbox{if \quad $G_c(\widehat{p})< \HV(G^+) (\theta )$.}
\end{cases}
\end{eqnarray*}
We have thus shown that 
\begin{equation*}  
\HV (G_c)(\theta )=
  \begin{cases}
    \HV (G^+)(\theta )&\text{if}\quad \theta >{\theta} _+\\
    G_c(\widehat{p})&\text{if }\quad{\theta} _{-}\leqslant \theta \leqslant{\theta} _+\\ 
    \HV (G^-)(\theta )&\text{if}\quad \theta < {{\theta} _{-}},
  \end{cases}
\end{equation*}
where ${\theta}_+$ (resp.\ ${\theta}_-$) is the unique
solution in $[\widehat{p}, c)$ (resp.\ $(-c,\widehat{p}]$) of the
equation \[
\HV (G^+)(\theta) =G_c(\widehat{p})\qquad (\hbox{resp. \ \ $\HV (G^-)(\theta) =G_c(\widehat{p})$}\ ).
\]
Indeed, $\HV (G^+)(\widehat{p})\geqslant  G^+(\widehat{p})=G_c(\widehat{p})>\beta=\HV (G^+)(c)$, hence the existence and uniqueness of such 
a $\theta_+$ follows from the convexity of $\theta\mapsto\HV (G^+)(\theta)$. The reasoning for $\theta_-$ is analogous.

\section{Proof of Theorem~\ref{teo main 2}}\label{sez teo main 2}

Throughout this section we assume that
$V:\R\times\Omega\to [0,+\infty)$ is a stationary potential satisfying
(V1)--(V2). We start with a proposition, which is the key observation
needed for the proof of Theorem~\ref{teo main 2}. This proposition states
that in our  setting {\em homogenization} commutes with
{\em convexification} (i.e.\ taking the convex envelope of the momentum
part of the original Hamiltonian).

Given a function $h:\R\to\R$, we shall denote by $\conv(h)$ its (lower) convex
envelope
\[\conv(h)(p):=\sup\{g(p):\,g\text{ is convex and }\forall
  x\in\R,\ g(x)\le h(x)\},\quad \forall p\in\R.\]
\begin{prop}\label{commute}
Let $c_+\ge c_-$ and $G^\pm\in{\cal H}(\gamma,\alpha_0,\alpha_1)$ be non-negative convex functions such that  
$G^-(c_-)=G^+(c_+)=0$ and 
 \[ 
(G^+\wedge G^-)(p)=G^-(p)\quad\hbox{for $p\leq c_-$},\quad\ (G^+\wedge G^-)(p)=G^+(p)\quad\hbox{for $p\geq c_+$}.  
\]
Then
\[
\HV(\conv(G^+\wedge G^-))=\conv(\HV(G^+)\wedge \HV(G^-)).
\] 
\end{prop}
In turn, the above proposition is a simple consequence of the following observation.
\begin{lemma}\label{coincide}
  Let $G,G^+\in{\cal H}(\gamma,\alpha_0,\alpha_1)$ be non-negative
  convex functions such that $G(0)=G^+(0)=0$.
  \begin{enumerate}[(i)]
\item   If $G(p)=G^+(p)$ for all $p\ge 0$,\  then \quad  $\HV(G^+)(\theta)=\HV(G)(\theta)\quad\text{for all }\theta\ge 0.$\medskip
\item  If $G(p)=G^-(p)$ for all $p\le 0$,\  then \quad $\HV(G^-)(\theta)=\HV(G)(\theta)\quad\text{for all }\theta\le 0.$
\end{enumerate}
\end{lemma}

\begin{proof}
  We shall prove only item (i), since the argument for (ii) is
  symmetric. Fix an arbitrary $\theta\geqslant 0$ such that
  $\lambda:=\HV(G)(\theta)>\beta$. Then there is a corrector
  $F_\theta(x,\omega)$ for
   \[
   a(x,\omega) u''+G(\theta+u')+\beta V(x,\omega)=\lambda\qquad\hbox{in $\R$.}
 \]
 We claim that $F_\theta(x,\omega)$ is also a corrector for 
  \begin{equation}\label{coin1}
   a(x,\omega)u''+G^+(\theta+u')+\beta V(x,\omega)=\lambda\qquad\hbox{in $\R$}.
  \end{equation}
  This follows immediately from derivative
  estimates of Proposition~\ref{prop inequality corrector}. Indeed, by
  this proposition,
  \[a^-_\lambda\le \theta+F'_\theta(x,\omega)\le b^-_\lambda,\] where
  $a^-_\lambda,b^-_\lambda>0$, $G(a^-_\lambda)=\lambda-\beta$, and
  $G(b^-_\lambda)=\lambda$.  Since $G^+(p)=G(p)$ for all $p\ge 0$, we
  conclude that
  \[
    G^+(\theta+F'_\theta(x,\omega)) = G(\theta+F'_\theta(x,\omega))
    \qquad\hbox{in $\R$}
  \] 
  in the viscosity sense. The existence of a corrector for equation
  \eqref{coin1} with $\lambda=\HV(G)(\theta)$ implies that
  $\HV(G^+)(\theta)=\lambda=\HV(G)(\theta)$. We conclude that
  \[\HV(G^+)(\theta)=\HV(G)(\theta)\quad \text{on the set 
      $\{\theta\geqslant 0\,|\,\HV(G)(\theta)>\beta\}$.}\] Exchanging the roles
  of $G$ and $G^+$ we also have that
  \[\HV(G)(\theta)=\HV(G^+)(\theta)\quad\text{on the set
      $\{\theta\geqslant 0\,|\,\HV(G^+)(\theta)>\beta\}$.}\] The last two
  statements in combination with the fact that
  $(\HV(G)\wedge\HV(G^+))(\theta)\ge \beta$ for all $\theta\in\R$
  complete the proof of the lemma.
\end{proof}

\begin{proof}[Proof of Proposition~\ref{commute}]
  Since $\beta=\min_{\theta\in\R} \HV(G^\pm)(\theta)$, we
  have that
\begin{equation}
  \label{form}
  \conv(\HV(G^+)\wedge \HV(G^--))(\theta)=
  \begin{cases}
    \HV(G^+)(\theta),&\text{if }\theta\ge c_+;\\\beta,&\text{if }-c_-\le \theta\le c_+;\\ \HV(G^-)(\theta),&\text{if }\theta\le -c_-.
  \end{cases}
\end{equation}
On the other hand, \[G(p):=\conv(G^+\wedge G^-)(p)=
  \begin{cases}
    G^+(p),&\text{if }p\ge c_+;\\0,&\text{if }-c_-\le p \le c_+;\\G^-(p),&\text{if }p\le -c_-.
  \end{cases}
\] Applying the first part of Lemma~\ref{coincide} to functions $G(\cdot +c_+),G^+(\cdot+c_+)$ and the second to functions $G(\cdot-c_-),G^-(\cdot-c_-)$ we infer that \[\HV(G)(\theta)=
  \begin{cases}
    \HV(G^+)(\theta),&\text{if }\theta\ge c_+;\\\HV(G^-)(\theta),&\text{if }\theta\le -c_-.
  \end{cases}
\]
By Proposition \ref{prop Lambda}-(iv), we know that $\HV(G)(c_+)=\HV(G)(c_-)=\beta$. Combining this with the fact that $\HV(G)$ is convex and $\HV(G)(\theta)\ge \beta$ for all $\theta\in\R$, see Proposition \ref{prop lower bound}, we get that $\HV(G)$ coincides with the right hand side of \eqref{form}. This finishes the proof.
\end{proof}

  \begin{proof}[Proof of Theorem~\ref{teo main 2}]
    First of all, we note that by Theorem 1.1 the right hand side of
    \eqref{eh0n} is equal to \eqref{eh0nl}.

\medskip
    
 \noindent{\em Upper bound.} Since $(G_0\wedge G_1\wedge\dots\wedge G_n)(p)\le G_{i-1,i}(p)$ for
 all $p\in\R$ and $i\in\{1,2,\dots,n\}$, by comparison, the left hand side of
 \eqref{eh0n} does not exceed the right hand side.

\medskip
 
\noindent{\em Lower bound.} We introduce a piece of notation first. For
all $i<j$, let us set $G_{ij}:=G_i\wedge G_j$ and denote by 
$\widehat{p}_{ij}\in (c_i,c_j)$ a  solution of the equation $G_i(p)=G_j(p)$. Note that 
\[
(G_i\wedge G_j)(p)=G_i(p)\quad\hbox{if }{p\le \widehat{p}_{ij},}
\qquad
(G_i\wedge G_j)(p)=G_j(p)\quad\hbox{if }{p\ge \widehat{p}_{ij},}
\]
and
\[G_{ij}(\widehat{p}_{ij})=G_i(\widehat{p}_{ij})=G_j(\widehat{p}_{ij})=\max_{p\in[c_i,c_j]}G_{ij}(p).\]
Set $G_{00}:= G_0$ and $G_{nn}:= G_n$. By comparison, for each $i\in\{1,2,\dots,n\}$
 \begin{equation}
   \label{lb1}
   \HV(G_0\wedge G_1\wedge\dots\wedge G_n)\ge \HV(\conv(G_{0,i-1})\wedge \conv(G_{in})).
 \end{equation}
 Next we shall write the formula for $\HV(G_{i-1,i})$ from Theorem 1.1
 in a way which covers both weak and strong potential cases. For
 $i\in\{1,2,\dots,n\}$
 \begin{equation}
   \label{lb0}
   \HV(G_{i-1,i})(\theta)=
   \begin{cases}
     \HV(G_{i-1})(\theta),&\text{if }\theta<\theta_{i-1,i}^-;\\ G_{i-1,i}(\widehat{p}_{i-1,i})\vee \beta,&\text{if }\theta_{i-1,i}^-\le \theta\le \theta_{i-1,i}^+;\\ \HV(G_i)(\theta),&\text{if }\theta>\theta_{i-1,i}^+,
   \end{cases}
 \end{equation}
where $\theta_{i-1,i}^-$ (resp.\ $\theta_{i-1,i}^+$) is the smallest (resp.\ largest) solution in $[c_{i-1},\widehat{p}_{i-1,i}]$ (resp.\ $[\widehat{p}_{i-1,i},c_i]$) of the equation \[\HV(G_{i-1})(\theta)=G_{i-1,i}(\widehat{p}_{i-1,i})\vee \beta\quad (\text{resp. }\HV(G_i)(\theta)=G_{i-1,i}(\widehat{p}_{i-1,i})\vee \beta).\] In the strong potential case we simply have $\theta_{i-1,i}^-=c_{i-1}$ and $\theta_{i-1,i}^+=c_i$ (see Figure~\ref{fig1}).

 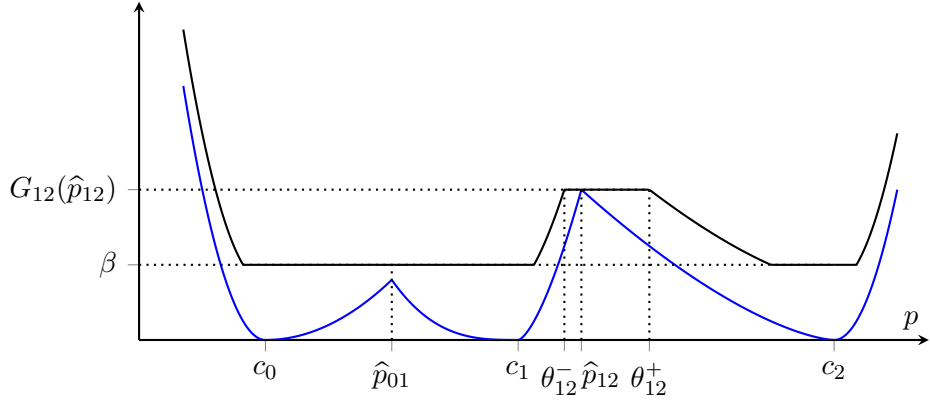
\begin{figure}[h!]
   \begin{tikzpicture}
   \begin{axis}[scale only axis=true,
         width=0.7\textwidth,
         height=0.3\textwidth,
         axis x line=middle, xlabel={$p$}, axis y line=left,
     ylabel={$\ $}, tick align=outside, samples=100,
     xtick={-4,-2,0,0.7310,1,2.076,5},
     xticklabels={$c_0$,$\widehat{p}_{01}$,$c_1$,$\theta_{12}^-\ $,$\ \quad\widehat{p}_{12}$,$\theta_{12}^+$, $c_2$}, xmin=-6.0,
     xmax=6.5, ytick={1,2}, yticklabels={$\beta$, $G_{12}(\widehat{p}_{12})$}, ymax=4.5, thick]
     \addplot[domain=-5.3:-4.0, no marks, blue] (\x,{2*(\x+4)^2});
     \addplot+[domain=-4.0:-2.0, no marks, blue] (\x,{0.2*(\x+4)^2});
     \addplot+[domain=-2.0:0.0,  no marks, blue] (\x,{(-\x)^3*0.1});
     \addplot+[domain=0.0:1.0,  no marks, blue](\x,{2*(\x)^(1.5)});
     \addplot+[domain=1.0:5.0, no marks, blue] (\x,{0.25*(abs(5-\x))^(1.5)});
     \addplot+[domain=5.0:6.0, no marks, solid, blue](\x,{2*(\x-5)^2});
     \addplot[domain=-5.3:-4.0, no marks, black] (\x,{max(2*(\x+4)^2+0.75,1)});
     \addplot+[domain=-4.0:-2.0, no marks, solid, black] (\x,{min(max(0.2*(\x+4)^2+0.75,1),1)});
     \addplot+[domain=-2.0:0.0,  no marks, solid, black] (\x,{min(max((-\x)^3*0.1+0.75,1),1)});
     \addplot+[domain=0.0:1.0,  no marks, solid, black](\x,{min(max(2*(\x)^(1.5)+0.75,1),2)});
     \addplot+[domain=1.0:5.0, no marks, black] (\x,{min(max(0.25*(abs(5-\x))^(1.5)+0.75,1),2)});
     \addplot+[domain=5.0:6.0, no marks, solid, black](\x,{max(1,2*(\x-5)^2+0.75,1)});
     \addplot[mark=none, dotted] coordinates {(-2, 0) (-2, 1)};
     \addplot[mark=none, dotted] coordinates {(0.7310, 0) (0.7310, 2)};
     \addplot[mark=none, dotted] coordinates {(1, 0) (1, 2)};
     \addplot[mark=none, dotted] coordinates {(2.076, 0) (2.076, 2)};
     \addplot[mark=none, dotted] coordinates {(-6, 1) (5, 1)};
     \addplot[mark=none, dotted] coordinates {(-6, 2) (2, 2)};
 \end{axis}
 \end{tikzpicture}
 \caption{The original Hamiltonian $G_0(p)\wedge G_1(p)\wedge G_2(p)$ is depicted in blue and the effective Hamiltonian is in black. Note that $\theta^-_{01}=c_0$ and $\theta^+_{01}=c_1$.}
 \label{fig1}
 \end{figure}

In the same way, for each $i\in\{1,2,\dots,n\}$ and $\theta\in[c_{i-1},c_i]$
we get that
\begin{equation}
  \label{lb2}
  \HV(\conv(G_{0,i-1})\wedge \conv(G_{in}))=
  \begin{cases}
    \HV(\conv(G_{0,i-1}))(\theta),&\text{if }\theta<\theta_{i-1,i}^-;\\
    G_{i-1,i}(\widehat{p}_{i-1,i})\vee \beta,&\text{if }\theta_{i-1,i}^-\le \theta\le \theta_{i-1,i}^+;\\
    \HV(\conv(G_{in}))(\theta),&\text{if }\theta>\theta_{i-1,i}^+.
  \end{cases}
\end{equation}
We emphasize that $\theta_{i-1,i}^\pm$ which appear in \eqref{lb2}
are the same as in \eqref{lb0}. Indeed, by the definition of $\theta_{i-1,i}^-$ and Proposition~\ref{commute}, $\theta_{i-1,i}^-\ge c_{i-1}$ and  
\begin{align*}
  \HV(\conv(G_{0,i-1}))(\theta)=\conv(\HV(G_{0}) \wedge \HV (G_{i-1}))(\theta) =\HV(G_{i-1})(\theta)\ \ \text{for $\theta\ge c_{i-1}$}.
\end{align*}
Similarly, $\theta_{i-1,i}^+\le c_i$ and
\[
\HV(\conv(G_{in}))(\theta)=\conv(\HV(G_{i}) \wedge \HV (G_{n}) )(\theta)
=
\HV(G_i)(\theta) \ \  \text{for $\theta\le c_i$.}
\]
These formulas together with \eqref{lb0} and \eqref{lb2}
imply that for all $\theta\in [c_{i-1},c_i],\ i\in\{1,2,\dots,n\}$
\begin{equation}\label{lb21}
  \HV(\conv(G_{0,i-1})\wedge \conv(G_{in}))(\theta)=\HV(G_{i-1,i})(\theta).
\end{equation}
From \eqref{lb1}, \eqref{lb2}, and \eqref{lb21} we conclude that 
\begin{align}
  \label{lb31}
  \HV(G_0\wedge G_1\wedge\dots&\wedge G_n)(\theta)\ge \max_{j\in\{1,2,\dots,n\}}\HV(\conv(G_{0,j-1})\wedge \conv(G_{jn}))(\theta)\\ \label{lb32}&=
  \begin{cases}
    \HV(G_{00}),&\text{if }\theta<c_0;\\
    \HV(G_{i-1,i}),&\text{if }  c_{i-1}\le \theta\le c_i,\ i\in\{1,2,\dots,n\};\\
    \HV(G_{nn}),&\text{if } \theta>c_n.
  \end{cases}
\end{align}
Indeed, for all $\theta\le c_0 <\theta^-_{01}$ and all $j\in\{1,2,\dots,n\}$
\begin{multline*}
  \HV(\conv(G_{0,j-1})\wedge \conv(G_{jn}))(\theta)\overset{\eqref{lb2}}{=}\HV(\conv(G_{0,j-1}))(\theta)\\ \overset{\text{Prop.~\ref{commute}}}{=}\conv(\HV(G_{0,j-1}))(\theta)\overset{\text{Th.}~\ref{teo main 1}}{=}\HV(G_0)(\theta)=\HV(G_{00})(\theta).
\end{multline*}
This proves that the right hand side of \eqref{lb31} is equal to the
first line of \eqref{lb32} when $\theta\leq c_0$. Similar argument
establishes the equality for $\theta\geq c_n$. Next, combining \eqref{lb21} with
the fact that for all $i,j\in\{1,2,\dots,n\}$ such that $j\ne i$
\begin{equation}
  \label{lb34}
  \HV(\conv(G_{0,j-1})\wedge \conv(G_{jn}))(\theta)=\beta\le \HV(G_{i-1,i})(\theta),\quad \theta\in[c_{i-1},c_i],
\end{equation}
we obtain the equality between the right hand side of \eqref{lb31}
and \eqref{lb32}.  Finally, we notice that \eqref{lb32} coincides with
\eqref{eh0nl}. This completes the proof.
  \end{proof}

  \appendix  
\section{PDE results}\label{pderes}

In this appendix we state and prove some PDE results we need for our study. 
We introduce the following list of conditions on the ingredients of the parabolic and stationary Hamilton--Jacobi equations 
we will consider, for fixed constants $\alpha_0,\alpha_1, \kappa>0$ and $\gamma>1$:
\begin{itemize}
\item[(AV)] \quad $\sqrt{a}, V:\R\to [0,1]$ are $\kappa$--Lipschitz continuous;\medskip
\item[(G1)] \quad $\alpha_0|p|^\gamma-1/\alpha_0\leqslant G(p)\leqslant\alpha_1(|p|^\gamma+1)$\qquad for all $x,p\in\R$;\medskip
\item[(G2)] \quad
  $|G(p)-G(q)|\leqslant\alpha_1\left(|p|+|q|+1\right)^{\gamma-1}|p-q|$\qquad for all $p,q\in\R$;\medskip
\item[(G3)] \quad $G(0)=0$;\medskip
\item[(G4)] \quad $G(\cdot)$ is convex;\medskip
\item[(G5)] \quad $G(p)\geqslant 0$ for every $p\in\R$. 
\end{itemize}

In what follows, we will denote by $\D{LSC}(X)$ and $\D{USC}(X)$ the space of 
lower and upper semi-continuous real functions on the topological space $X$, respectively.

\subsection{Parabolic equation}\label{sez appendix parabolic}

We consider the parabolic equation 
\begin{equation}\label{eq parabolic}
\partial_{t }u=a(x)\partial^2_{xx} u +G(\partial_x u)+  \beta V(x,\omega)  \qquad \hbox{in $(0,+\infty)\times\R$}.
\end{equation}

We have the following comparison result.

\begin{prop}\label{prop Lip comp}
Assume condition (AV) and let $G\in\CC(\R)$. 
Let $v\in\D{USC}([0,T]\times\R)$,\, $w\in\D{LSC}([0,T]\times\R)$ be, respectively, a {sub- and a super- solutions} of 
\eqref{eq parabolic} satisfying 
\begin{equation}\label{hyp 2}
\limsup_{\substack{|x|\to +\infty}}\ \sup_{t\in [0,T]}\frac{v(t,x)}{1+|x|}\leqslant 0 
\leqslant 
\liminf_{\substack{|x|\to +\infty}}\ \inf_{t\in [0,T]}\frac{w(t,x)}{1+|x|}.
\end{equation}
Let us furthermore assume that either $\partial_x v$ or $\partial_x w$
belongs to 
$L^\infty\left((0,T)\times \R\right)$.  Then
\[
v(t,x)-w(t,x)\leqslant \sup_{\,\R}\big(v(0,\cdot)-w(0,\cdot)\big)\quad\hbox{for every  $(t,x)\in\cTcyl$.}  
\]
\end{prop}

The proof is standard, see for instance \cite[Proposition 2.3]{DK17} and \cite[Appendix A]{D19}. The next result shows that equation \eqref{eq parabolic} is 
well posed in $\D{UC}(\ccyl)$.

\begin{teorema}\label{teo parabolic eq}
  Assume conditions (AV) and (G1)-(G2). Then for every $g\in\D{UC}(\R)$ there exists a unique solution $u\in\D{UC}({\ccyl})$ of \eqref{eq parabolic} 
  satisfying $u(0,\cdot)=g$ on $\R$. 
\end{teorema}

We also need the following Lipschitz bounds for solutions to \eqref{eq parabolic} with linear initial data. 
We refer to \cite[Theorem 2.8]{DK17} for proofs.

\begin{prop}\label{prop Lip estimates}
Assume conditions (AV) and (G1)-(G2). For every $\theta\in\R$, the unique solution $u_\theta$ of \eqref{eq parabolic} in 
$\D{UC}(\ccyl)$ with initial condition $u(0,x)=\theta x$ is $\kappa(\theta)$--Lipschitz continuous in $\ccyl$, for some locally bounded 
functions $\kappa:\R\to [0,+\infty)$. 
\end{prop}

\subsection{Stationary viscous Hamilton--Jacobi equation}\label{sez appendix stationary}
Let us consider the equation 
\begin{equation}\label{stationary viscous HJ equation}
-a(x)u''+G(u')+\beta V(x)=\lambda\qquad\hbox{in $I$,}
\end{equation}
where $I$ is an open subset of $\R$ and $\lambda>\beta>0 $. We will be interested in the cases when $I=\R\setminus\{y\}$, $I=(-\infty,y)$, $I=(y,+\infty)$. 

The following comparison principle holds:

\begin{teorema}\label{teo comparison}
Assume conditions (AV) and (G1)--(G4). Let $y,\theta\in\R$, $\lambda>\beta>0 $ and $u\in\D{LSC}(\R),v\in\D{USC}(\R)$ be, respectively, a super- and sub- solution of 
\begin{equation}\label{eq comparison}
-a(x)u''+G(\theta +u')+\beta V(x)=\lambda\qquad\hbox{in $\R\setminus\{y\}$,}
\end{equation}
satisfying 
\[
\limsup_{|x|\to +\infty} \frac{v(x)}{1+|x|}
\leqslant 
0
\leqslant 
\liminf_{|x|\to +\infty}  \frac{u(x)}{1+|x|}.
\]  
The following statements hold:
\begin{enumerate}[(i)]
\item if $\theta > 0$, then \quad $(v-u)(x)\leqslant (v-u)(y)$\quad  for every $x\geqslant y$;\medskip
\item if $\theta < 0$, then \quad $(v-u)(x)\leqslant (v-u)(y)$ \quad for every $x\leqslant y$;\medskip
\item if $\theta = 0$ and $v\in\Lip(\R)$, then \quad $(v-u)(x)\leqslant (v-u)(y)$ \quad for every $x\in \R$.
\end{enumerate}
\end{teorema}

\begin{proof}
Without loss of generality, we can assume $y=0$ and $v(0)=u(0)=0$. Let us set $\tilde v(x):=\theta x+v(x)$, $\tilde u(x):=\theta x +u(x)$ and, for 
$\mu\in (0,1)$  $\tilde v_\mu(x):=\mu \tilde v(x)=\mu \tilde v(x)+(1-\mu)0$. Since the 
function $v_0\equiv 0$ is a strict subsolution of \eqref{stationary viscous HJ equation} in $\R$ (due to the fact that $\lambda>\beta>0 $ and $G(0)=0$), by convexity of $G$ we 
infer that $\tilde v_\mu$ is a strict subsolution to \eqref{stationary viscous HJ equation} in $\R\setminus\{0\}$, see \cite[Lemma 2.4]{AT}, i.e. $\tilde v_\mu$ satisfies 
the following inequality in the viscosity sense for some $\delta>0$:
\begin{equation}\label{eq strict}
-a(x)\tilde v_\mu''+G(\tilde v_\mu')+\beta V(x)< \lambda-\delta\qquad\hbox{in $\R\setminus\{0\}$.}
\end{equation}
Now, if $\theta>0$, we have 
\begin{eqnarray*}
\limsup_{x\to +\infty} \frac{\tilde v_\mu (x)-\tilde u(x)}{1+|x|}
&&\leqslant 
\limsup_{x\to+\infty} \frac{-(1-\mu) \theta x +\mu v(x)-u(x)}{1+|x|}\\
&&\leqslant
\lim_{x\to+\infty} \frac{-(1-\mu) \theta x}{1+|x|}
+
\limsup_{x\to+\infty} \frac{\mu v(x)}{1+|x|}
\,-\,
\liminf_{x\to +\infty} \frac{u(x)}{1+|x|}\\
&&\leqslant 
\lim_{x\to +\infty} -(1-\mu)\theta \frac{x}{|x|}
=
-(1-\mu)\theta <0,
\end{eqnarray*}
in particular $(\tilde v_\mu-\tilde u)(x)\to -\infty$ as $x\to +\infty$. This means that the open set $I_\mu:=\{x>0\,\mid\,\tilde v_\mu-\tilde u>0\,\}$ is bounded, so we can apply 
\cite[Theorem 2.2]{AT} to get
\[
\sup_{I_\mu}\left(\tilde v_\mu-\tilde u\right)\leqslant \sup_{\partial I_\mu} \left(\tilde v_\mu-\tilde u\right)=0,
\]
where in the last equality we have also used the fact that $v_\mu(0)-u(0)=0$. From this we infer that 
\[
\tilde v_\mu(x)-\tilde u(x)= (\mu v(x)-u(x))-(1-\mu)\theta x \leqslant 0\quad\hbox{for all $x\geqslant 0$.}
\]
By sending $\mu\nearrow 1$ we get \quad $v(x)-u(x)\leqslant 0=v(0)-u(0)$ \quad for all $x\geqslant 0$, as asserted. 

If $\theta<0$, then, arguing as above, we get
\[
\limsup_{x\to -\infty} \frac{\tilde v_\mu (x)-\tilde u(x)}{1+|x|}
\leqslant
\lim_{x\to -\infty} -(1-\mu)\theta \frac{x}{|x|}
=
(1-\mu)\theta <0,
\]
in particular $(\tilde v_\mu-u)(x)\to -\infty$ as $x\to -\infty$. This means that the open set $I_\mu:=\{x<0\,\mid\,\tilde v_\mu-u>0\,\}$ is bounded. 
By arguing as in the previous case, we conclude that \quad $v(x)-u(x)\leqslant 0=v(0)-u(0)$ \quad for all $x\leqslant 0$. 

If $\theta=0$, then $\tilde v =v$ and $\tilde u=u$. Let us write $v_\mu$ in place of $\tilde v_\mu$ and set $v_{\mu}^\eps(x):=v_\mu(x)-\eps\sqrt{1+x^2}$ for every $x\in\R$. Because of \eqref{eq strict} and the fact that $v_\mu\in\Lip(\R)$, an easy computation shows that for $\eps>0$ small enough $v_\mu^\eps$ is a strict subsolution to \eqref{stationary viscous HJ equation} in $\R\setminus\{0\}$, i.e. satisfies
\eqref{eq strict}. We have 
\begin{eqnarray*}
\limsup_{|x|\to +\infty} \frac{v^\eps_\mu (x)- u(x)}{1+|x|}
&&\leqslant 
\limsup_{|x|\to+\infty} \frac{\mu v(x)-\eps\sqrt{1+x^2}}{1+|x|}
\,-\,
\liminf_{|x|\to +\infty} \frac{u(x)}{1+|x|}
\leqslant 
-\eps<0,
\end{eqnarray*}
in particular $(v^\eps_\mu-u)(x)\to -\infty$ as $|x|\to +\infty$. This means that the open set $I_\mu:=\{x\in\R\,\mid\, v^\eps_\mu-u>0\,\}$ is bounded, so we can apply 
\cite[Theorem 2.2]{AT} and argue as above to infer 
\[
v^\eps_\mu(x)-u(x)= (\mu v(x)-u(x))-\eps\sqrt{1+x^2} \leqslant 0\quad\hbox{for all $x\in\R$.}
\]
By sending $\eps\searrow 0 $ and $\mu \nearrow 1$ we conclude that  $v(x)-u(x)\leqslant 0=v(0)-u(0)$ \quad for all $x\in\R$.
\end{proof}

As a corollary we infer

\begin{cor}\label{cor unique solution}
Let $\theta\in\R\setminus\{0\}$ and $u_1,u_2$ be sublinear solutions of
\[
-a(x)u''+G(\theta +u')+\beta V(x)=\lambda\qquad\hbox{in $\R$,}
\]
where $\lambda>\beta>0 $. 
Then $u_1-u_2$ is constant on $\R$. 
\end{cor}
\begin{proof}
To fix ideas, let us assume $\theta>0$. Let us fix $y\in\R$. By applying Theorem \ref{teo comparison}, we get 
\[
(u_1-u_2)(x)\leqslant (u_1-u_2)(y)\qquad\hbox{for all $x\geqslant y$,}
\]
and, symmetrically, 
\[
(u_2-u_1)(x)\leqslant (u_2-u_1)(y)\qquad\hbox{for all $x\geqslant y$.}
\]
We conclude that, for every $y\in \R$, we have 
\[
(u_1-u_2)(x)=(u_1-u_2)(y)\qquad\hbox{for all $x\geqslant y$.}
\]
This readily implies that $u_1-u_2$ is constant on $\R$. The argument in the case $\theta< 0$ is analogous.
\end{proof}
%
%

We also need the following version of the comparison principle. 

\begin{teorema}\label{teo comparison bis}
Assume conditions (AV) and (G1)--(G5).
Let $\lambda>\beta>0 $, $y\in\R$ and $I$ be either $I=(-\infty,y)$ or $I=(y,+\infty)$. 
Let $u\in\Lip(I)$ and $v(x):=\kappa |x-y|$, $\kappa>0$, be, respectively, a super- and sub- solution of 
\begin{equation}\label{eq comparison2}
-a(x)u''+G(u')+\beta V(x)=\lambda\qquad\hbox{in $I$}
\end{equation}
with 
\[
\liminf_{x\in I, |x|\to +\infty}  \frac{u(x)}{1+|x|}\geqslant 0.
\]
Then 
\[
(v-u)(x)\leqslant (v-u)(y)\qquad\hbox{for all $x\in I$.}
\]
\end{teorema}

This comparison principle can be easily proved arguing as in the proof of Theorem \ref{teo comparison}--(iii) with the aid of the following lemma.

\begin{lemma}\label{lemma sublinear approximation}
Let $y\in\R$ and let $I$ be either 
$I=(-\infty,y)$ or $I=(y,+\infty)$. Let $v(x):=\kappa |x-y|$, $\kappa>0$, be a subsolution to \eqref{eq comparison2} where $\lambda>\beta>0 $. 
Then $v(x)=\sup_{w\in{\sol}^-(v)} w(x)$, $x\in I$, where we have denoted by ${\sol}^-(v)$ the set of bounded subsolutions $w:I\to\R$ of 
\eqref{eq comparison2} satisfying $w\leqslant v$ in $I$. 
\end{lemma}

\begin{proof}
Without any loss of generality, we can assume $y=0$. 
For every $\mu\in [0,1)$ we set $v_\mu(x):=\mu v(x)$. Since the 
function $v_0\equiv 0$ is a strict subsolution of \eqref{eq comparison2} in $I$ (due to the fact that $\lambda>\beta>0 $ and $G(0)=0$), by convexity of $G$ we 
infer that $v_\mu=\mu v+(1-\mu)v_0$ is a strict subsolution to \eqref{eq comparison2} in $I$, i.e. $v_\mu$ satisfies 
the following inequality in the viscosity sense for some $\delta>0$:
\begin{equation}\label{eq strict2}
- a(x)v_\mu''+G(v_\mu')+\beta V(x) =G(v_\mu')+\beta V(x)< \lambda-\delta\qquad\hbox{in $I$.}
\end{equation}
Since $v(x)=\sup_{\mu \in (0,1)} v_\mu(x)$, we infer that it is suffices to prove the assertion by additionally assuming that $v(x)=\kappa |x|$ satisfies 
\eqref{eq strict2} for some $\delta>0$. For fixed $n\in\N$ and $\eps>0$, we define $\varphi_{\eps, n}(u):=\int_0^u g_{\eps, n}(t)\, dt$\ for all $u\geqslant 0$, where 
\begin{eqnarray*}
g_{\eps, n}(t):=
\begin{cases}
1 & \hbox{if $0\leqslant t\leqslant \kappa n$}\smallskip\\
f(\eps(t-n)) & \hbox{if $t>\kappa n$,}
\end{cases}
\end{eqnarray*}
with $f(s):=\e^{-s^2}$. Let $C>0$ be such that $-C\leqslant f'(s) <0$ for every $s\geqslant 0$. An easy check shows that $\varphi_{\eps, n}$ is bounded and of class $\D{C}^2$ 
on $[0,+\infty)$,   
$0<\varphi_{\eps, n}'\leqslant 1$ and $\varphi_{\eps, n}''\geqslant -\eps C$ in $[0,+\infty)$. 
Let us set $v_{\eps, n}(x):=\varphi_{\eps, n}(v(x))$. Then 
\begin{eqnarray}\label{eq properties v_n}
v_{\eps, n}\leqslant v\quad\hbox{in $I$},\ \  v_n=v\ \ \hbox{in $I\cap [-n,n]$,}\ \ 
v_{\eps, n}''(x)=\kappa^2 \varphi_{\eps, n}''(v(x))\geqslant -C\kappa^2\eps\ \ \hbox{in $I$}.
\end{eqnarray}
Also notice that $v_{\eps, n}'(x)=\varphi_{\eps, n}'(v(x))v'(x)$ and $v'(x)$ have the same sign (either positive or negative) and 
$|v_{\eps, n}'| \leqslant |v'|$ in $I$. Since $G$ is non--increasing in $(-\infty, 0]$ and non--decreasing on $[0,+\infty)$, we infer that $G(v'_{\eps, n}(x))\leqslant G(v'(x))$ for every 
$x\in I$. So 
\[
-a(x)v_{\eps, n}''+G(v_{\eps, n}')+\beta V(x) 
\leqslant
C\kappa^2\eps +G(v')+\beta V(x) 
\leqslant
\lambda-\delta+C\kappa^2\eps\qquad\hbox{in $I$,}
\]
hence by choosing $\eps<\delta/(C\kappa^2)$ we get that $v_n:=v_{\eps, n}\in \sol^-(v)$. By taking into account 
\eqref{eq properties v_n}, we conclude that $v(x)=\sup_{n\in\N} v_n(x)$, which clearly implies the assertion.          
\end{proof}

 {\section{Hills and valleys}\label{hv}

   {In this section we discuss the relationship between the
    original hill and valley condition of \cite{KYZ20}} (see
  ($\Lambda$V) below) and its scaled version (V2), give examples of
  potentials which satisfy conditions \eqref{intro 01}, (V1) and
  ($\Lambda$V) and argue that potentials satisfying ($\Lambda$V) can
  be thought of as ``typical'' in the general stationary ergodic
  setting.

  \subsection{Comparison of ($\Lambda$V) and (V2)}\label{oss weak valley-and-hill}
  
   {In \cite{KYZ20} the authors considered the case $a\equiv 1/2$
    and imposed the following hill and valley condition on $V$:
  \begin{itemize}
  \item [($\Lambda$V)] For all $h\in(0,1)$ and $y>0$,
    \begin{itemize}
    \item [($\wedge$)]
      $\PP\left(\forall x\in[-y,y],\ V(x,\omega)\ge h\right)>0$;
    \item [($\vee$)] $\PP\left( \forall x\in[-y,y],\ V(x,\omega)\le h\right)>0$.
    \end{itemize}
    \end{itemize}
    First of all we note that ($\wedge$) can be
    stated equivalently as follows.
\begin{itemize}
\item [($\wedge'$)] for all $h\in(0,1)$ and $y>0$ there is a set
  $\Omega(h,y)$ of probability 1 such that for each
  $\omega\in\Omega(h,y)$ there is an $\ell\in\R$ such
  that $V(x,\omega)\ge h$ for all $x\in[\ell,\ell+2y]$.
\end{itemize}
Indeed, if we let
$A(h,y)=\bigcup_{\ell\in\R}\{\forall x\in[\ell,\ell+2y],\
V(x,\omega)\ge h\}$ (``there is an $h$-hill of length $2y$'') then
($\wedge$) implies that $\PP(A(h,y))>0$. But the event $A(h,y)$ is
invariant under translations. Therefore, by the ergodicity assumption
its probability is equal to 1, and ($\wedge'$) follows. Conversely,
note that
\begin{align*}
  (\wedge') \ \ \Longleftrightarrow\ \
  &\forall h\in(0,1),\ \forall y>0,\ \PP\left(A(h,y)\right)=1\\ \Longleftrightarrow\ \
  & \forall h\in(0,1),\ \forall y>0,\ \PP\left(\bigcup_{\ell\in\Z}
    \{\forall
    x\in[\ell,\ell+2y],\ V(x,\omega)\ge h\}\right)=1.
\end{align*}
Since
  \[\PP\left(\bigcup_{\ell\in\Z}\{\forall x\in[\ell,\ell+2y],\
      V(x,\omega)\ge h\}\right)\le \sum_{\ell\in\Z}\PP\left(\forall
      x\in[\ell,\ell+2y],\ V(x,\omega)\ge h\right)\]
and, by stationarity, 
  \[\forall\ell\in\R,\quad \PP\left(\forall x\in[\ell,\ell+2y],\
      V(x,\omega)\ge h\right)=\PP\left(\forall x\in[-y,y],\
      V(x,\omega)\ge h\right),\] we conclude that ($\wedge'$) implies
  ($\wedge$). The valley condition ($\vee$) admits a similar
  equivalent formulation.

  Now it is easily seen that in the uniformly elliptic case, i.e.\
  when $a(x,\omega)\ge a_0>0$ for all $x\in\R$ and $\omega\in\Omega$,
  ($\Lambda$V) is equivalent to (V2). We also point out that if
  $a(x_0,\omega)=0$ for some $x_0\in\R$, then (A) implies that
  $\int_I 1/a(x,\omega)\,dx=+\infty$ for every interval
  $I$ containing $x_0$. In particular, if $a\equiv 0$ and $V$ is
  continuous then (V2) reduces to \eqref{intro 01} and, thus, can be
  dropped altogether. }

\subsection{Examples} We start with probably the simplest class of
examples. They are based on stationary renewal processes and
i.i.d.\ sequences.
\begin{examp}\label{ex1}
  Let $(\Omega,{\cal F},\PP)$ be a probability space and $F$ be a
  distribution function on $\R$ with $F(0)=0$ and a finite mean
  $m:=\int_0^{+\infty} (1-F(t))\,dt$. Let
  $(X_k)_{k\in\Z\setminus\{-1,0\}}$ be a sequence of i.i.d.\ random
  variables with a common distribution $F$ and choose $(X_{-1},X_0)$
  to be independent from this sequence and
  distributed as follows: 
  \[\PP(X_{-1}\ge x,\,X_0\ge y)=\frac{1}{m}\int_{x+y}^{+\infty}
    (1-F(t))\,dt\quad\text{for all }x,y\ge 0.\] Define $S_0=X_0$,
  $S_n=\sum_{k=0}^nX_k$, $S_{-n}=-\sum_{k=-1}^{-n}X_k$ for all
  $n\in\N$. The process $(S_k)_{k\in\Z}$ is a stationary renewal
  process on $\R$ (see \cite[Theorem 9.1, Chapter
  9]{KT75})\footnote{If $F(x)=\mathbbm{1}_{[1,+\infty)}(x)$ then we get
    $S_k=X_0+k$, $k\in\Z$, where $X_0$ is uniform on $[0,1]$ and
    $X_{-1}=1-X_0$.}. Next we take an i.i.d.\ sequence $(\xi_k)_{k\in\Z}\in\{0,1\}^\Z$ of Bernoulli
  random variables with parameter $p\in(0,1)$ which is independent of
  $(X_k)_{k\in\Z}$, and define
  \[V(S_k,\omega)=\xi_k,\ \ 
      V(x,\omega)=\xi_k+(x-S_k)\frac{\xi_{k+1}-\xi_k}{S_{k+1}-S_k},\quad\text{for
      }S_k<x<S_{k+1},\ k\in\Z.
    \] 
In words, we let
  $V(S_k,\omega)=\xi_k$ and then linearly interpolate between the
  neighboring points. This stationary ergodic potential satisfies
  ($\Lambda$V) since for every $y>0$ there is an integer $n\ge 0$ such that
  \begin{enumerate}[(i)]
  \item the probability that $[-y,y]$ contains exactly $n$ renewal
    points is strictly positive and
  \item  $\PP(\xi_0=\xi_1=\dots=\xi_{n+1}=j)>0$ for $j\in\{0,1\}$ (we require
    require that at all $n$ points of the renewal process as well as
    at the points immediately preceeding and following these $n$
    points the potential takes the same value $j$).
  \end{enumerate}
 If, in addition, we assume that $F(x)=0$ for all
  $x<1/\kappa$ then $V$ also satisfies (V1).

  Both properties (V1) and ($\Lambda$V) are preserved if we replace Bernoulli
  distribution for $\xi_0$ with any probability distribution on
  $[0,1]$ as long as $\PP(\xi_0\in[0,h])>0$ and
  $\PP(\xi_0\in[1-h,1])>0$ for all $h\in(0,1/2)$. 
\end{examp}

\begin{oss}\label{reg}
  We also note that if we have a stationary ergodic potential $V_0$
  with values in $[0,1]$ which satisfies ($\Lambda$V) but not (V1), then we
  can take a $\kappa$-Lipschitz mollifier $f$ supported in $[0,1]$ and
  set $V(x,\omega)=\int_{\R}V_0(y,\omega)f(y-x)\,dx$. The resulting
  process $V$ will be stationary ergodic and will satisfy both (V1)
  and ($\Lambda$V).
\end{oss}

\begin{examp}\label{ex2}
  Let $(\Omega,{\cal F},\PP)$ be a probability space, $U$ be a uniform
  random variable on $[0,1]$ and $B_{\pm}=(B_\pm(t,\omega))_{t\ge 0}$
  be independent standard Brownian motions which are independent from
  $U$. Define a two-sided Brownian motion $B=(B(x,\omega))_{x\in\R}$
  starting from $U$ by \[B(x,\omega)=U(\omega)+
    \begin{cases}
      B_+(x,\omega),&\text{if }x\ge 0;\\B_-(-x,\omega),&\text{if }x<0.
    \end{cases}
  \] 
  We define $V(x,\omega)$ to be a
  Brownian motion with two-sided reflection, at $0$ and at $1$. The
  following informal description gives a pathwise construction of $V$
  from $B$ which avoids unsightly formulas. Imagine a fully
  transparent plane $\R^2$ and draw on it horizontal lines $y=k$ for
  all $k\in\Z$. Given a path of $B$, draw its graph
  $(x,B(x,\omega))\subset \R^2$, fold $\R^2$ as if it were a sheet of
  paper along the line $y=0$, then fold back along $y=1$ and $y=-1$,
  and continue folding back and forth until every line $y=k$,
  $k\in\Z$, of the original plane has been used. The resulting curve
  contained in $\R\times [0,1]$ is the graph of $V(x,\omega)$,
  $x\in\R$.

  The stationary ergodic process $V$ satisfies ($\Lambda$V), since (see, for
  example, \cite[Theorem (6.6), p.\,60]{Bass}) for every $h,R>0$
  \[\PP\left(\max_{x\in[-R,R]}
    |B(x,\omega)-1|<h\right)>0\quad\text{and}\quad\PP\left(\max_{x\in[-R,R]}
    |B(x,\omega)|<h\right)>0.\]

The paths $V(x,\omega)$, $x\in\R$, are almost surely not Lipschitz,
but, as we pointed out in Remark~\ref{reg}, a mollification will fix
this problem while preserving all other relevant properties.

  In this example Brownian motion with two-sided reflection can be
  replaced by a more general L\'evy process with two-sided
  reflection\footnote{A rigorous construction and properties of L\'evy
    processes with two-sided reflection can be found in
    \cite{AAGP}.} under
  suitable conditions on the support of its L\'evy measure.
\end{examp}
A very different class of examples was given in \cite[Example
1.3]{KYZ20}. Stochastic processes in this class have finite range of
dependence and, thus, are mixing with any rate. An example which is
not mixing can be constructed in the same way as in \cite[Example
1.3]{YZ19} by using points of a renewal process in Example~\ref{ex1}
instead of points of $\Z$.

\subsection{Discussion} Suppose that $V$ is stationary ergodic and
satisfies (V1) and \eqref{intro 01}. When does such process satisfy
($\Lambda$V)?  On $\Z$, an argument which shows that a very broad and natural
class of stationary ergodic processes satisfies ($\Lambda$V) was already put
forward in \cite[Example 1.2]{YZ19}. We state the following simple
sufficient condition for processes on $\R$ which can be checked in
many cases. Let $(\Omega,({\cal F}(x))_{x\in\R},{\cal F},\PP)$ be a
probability space,
${\cal F}(x_1)\subseteq {\cal F}(x_2)\subseteq{\cal F}$ for all
$x_1<x_2$ be a filtration, and $V$ be a stationary ergodic process
which satisfies (V1) and \eqref{intro 01} and is adapted to this
filtration. Observe that if for each $y>0$, $h\in(0,1/2)$ there is an
$n\in\N$ and $0<x_1<x_2<\dots< x_n<x_0+2y\le x_{n+1}$ such that
$\max_{1\le i\le n}|x_{i+1}-x_i|<1/(2\kappa)$,
  \begin{multline}
    \PP(|V(x_i,\omega)|<h/2,\ \forall
    i\in\{1,\dots,n\}\,|\,{\cal F}(0))>0\quad \text{and}\\
      \PP(|1-V(x_i,\omega)|<h/2,\ \forall
      i\in\{1,\dots,n\}\,|\,{\cal F}(0))>0,\label{md}
  \end{multline}
  then ($\Lambda$V) holds. All that \eqref{md} is asking for is that given the
  past of the process up to some fixed time (which by stationarity can
  be taken to be $0$), the probability that
  it finds itself ``on a hill'' (resp.\ ``in a valley'') at some
  future time $x_1>0$ and then happens to be there also at
  sufficiently many times $x_2,\dots,x_n$ within $1/(2\kappa)$ of each
  other is not equal to zero. Many natural stochastic processes will
  satisfy \eqref{md}. We remark that for Examples~\ref{ex1} and \ref{ex2}
  the property \eqref{md} holds.

  In conclusion, we mention an example of a process $V$ based on an
  i.i.d.\ sequence, satisfying (V1) and \eqref{intro 01} for which
  ($\Lambda$V) fails. It is discussed in detail in \cite{Sil20} where the author
  shows, in particular, that relative to the probability measure
  on $C(\R)$ corresponding to this process the sets of periodic and
  almost periodic functions are negligible. For reader's convenience
  we reproduce a version of this example here.
\begin{examp}\label{ex3}
  Without loss of generality we assume that $\kappa\ge 1$ and let
  $\varphi(x)$ be a $2\kappa$-Lipschitz\footnote{Function $\varphi$
    can have any Lipschitz constant as long as it is at least $2$. This will
    lead to an obvious adjustment of other parameters. Our choice
    gives the simplest formulas.} function such that (a)
  $0\le \varphi(x)\le 1$ for all $x\in \R$, (b)
  $\text{supp}\,\varphi(x)\subseteq [-1/2,1/2]$, and (c)
  $\varphi(x)=1$ iff $x=0$. In the setting of Example~\ref{ex1} take
  $F(x)=\mathbbm{1}_{[1,+\infty)}(x)$ so that $S_{k+1}-S_{k}=1$ with
  probability 1 and let $\zeta_k=2\xi_k-1\in\{-1,1\}$.  Define
  \begin{equation*}
    V(x,\omega)=\frac12\left(1+\sum_{k\in\Z}\zeta_k
    \varphi\left(x-S_k\right)\right).
  \end{equation*}
  The process $V$ clearly satisfies (V1) and \eqref{intro 01} but not
  ($\Lambda$V). On the other hand, no matter which $\varphi(x)$ as above we
  fix, it is enough to take any $F(x)$ such that $F(x)=0$ for all $x<1$,
  $1-F(x)>0$ for all $x>0$ and set
  \begin{multline*}
   V(x,\omega)=\\\frac12\left(1+\sum_{k\in\Z}\zeta_k
    \left(\varphi\left(\frac{x-S_k}{S_{k+1}-S_k}\right)\mathbbm{1}_{[0,+\infty)}(x-S_k)+\varphi\left(\frac{x-S_k}{S_k-S_{k-1}}\right)\mathbbm{1}_{(-\infty,0)}(x-S_k)\right)\right) 
  \end{multline*}
  to get the process which satisfies (V1) and ($\Lambda$V)\footnote{($\Lambda$V) is
    satisfied because $\varphi$ is continuous and
    $P(S_k-S_{k-1}>\ell, S_{k+1}-S_k>\ell)=(1-F(\ell))^2>0$ for all $k\in\Z\setminus\{-1,0\}$ no matter
    how large $\ell>1$ is.}. 
  \end{examp}
  Example~\ref{ex3} as well as the sufficient condition \eqref{md}
  indicate that essentially the only obstacle for the validity of ($\Lambda$V)
  is the ``rigidity'' of trajectories of $V$
  which is atypical for many classes of stationary ergodic processes,
  and even this ``rigidity'' can sometimes be rectified by making the
  ingredients ``more random''.
  }

\bibliography{viscousHJ}
\bibliographystyle{siam}
\end{document}